\documentclass[a4paper,12pt]{amsart}
\usepackage{amsmath,amssymb,amsfonts,amsthm,exscale,calc}
\usepackage[utf8]{inputenc}
\usepackage{latexsym,textcomp}
\usepackage{graphicx,epsf}
\usepackage{a4wide}
\usepackage[german, english]{babel}
\usepackage{hyperref}
\usepackage{enumerate}

\newtheorem{lemma}{Lemma}[section]
\newtheorem{theorem}[lemma]{Theorem}

\newtheorem{corollary}[lemma]{Corollary}

\theoremstyle{definition}
\newtheorem{definition}[lemma]{Definition}

\newtheorem*{remark}{Remark}
\newtheorem*{assumption}{Standing assumption}

\numberwithin{equation}{section}
\newcommand{\comment}[1]{}

\newcommand{\R}{{\mathbb R}}

\newcommand{\N}{{\mathbb N}}

\newcommand{\Qr}{Q_{\rho,V}}

\newcommand{\as}[1]{\left\langle #1\right\rangle}

\newcommand{\ow}[1]{\widetilde{ #1}}

\newcommand{\Hm}[1]{\leavevmode{\marginpar{\tiny%
$\hbox to 0mm{\hspace*{-0.5mm}$\leftarrow$\hss}%
\vcenter{\vrule depth 0.1mm height 0.1mm width \the\marginparwidth}%
\hbox to 0mm{\hss$\rightarrow$\hspace*{-0.5mm}}$\\\relax\raggedright
#1}}}

\newcommand{\vo}{{\rm vol}_g}
\newcommand{\Dm}{\Delta}
\newcommand{\dm}{\mathcal{D}(M)}
\newcommand{\dmp}{\mathcal{D}'(M)}

\newcommand{\cL}{\mathcal{L}_{\rho,V}}

\newcommand{\Lol}{L^1_{\rm loc}(M)}
\newcommand{\Lr}{L_{\rho,V}}
\newcommand{\Tr}{T^{\rho,V}_t}
\newcommand{\Gr}{G^{\rho,V}_\alpha}
\newcommand{\Qt}{\widetilde{Q}_{\rho,V}}

\begin{document}
\title[A generalized conservation property]
{A generalized conservation property for the heat semigroup on weighted manifolds}

\author[Masamune]{Jun Masamune}
\address{J. Masamune, 
Department of Mathematics, 
Faculty of Science,
Hokkaido University
Kita 10, Nishi 8, Kita-Ku, Sapporo, Hokkaido, 060-0810, Japan 
} \email{jmasamune@math.sci.hokudai.ac.jp}

\author[Schmidt]{Marcel Schmidt}
\address{M. Schmidt, Mathematisches Institut \\Friedrich Schiller Universit{\"a}t Jena \\07743 Jena, Germany } \email{schmidt.marcel@uni-jena.de}

\begin{abstract}
In this text we study a generalized conservation property for the heat semigroup generated by a Schrödinger operator with nonnegative potential on a weighted manifold. We establish Khasminskii's criterion for the generalized conservation property and discuss several applications.
\end{abstract}


\maketitle
\section{Introduction}

There are several ways to introduce Brownian motion on (weighted) Riemannian manifolds. One approach is through Markovian semigroups induced by self-adjoint realizations of the Laplacian. More precisely, if $\Delta_D$ is the self-adjoint realization of the  Laplacian with generalized Dirichlet boundary conditions, it follows from  standard theory on regular Dirichlet forms that there exists a unique Markov process $(B_t)_{t >0}$ on the manifold $M$ with life-time $\zeta$ such that the operator (or rather the semigroup it generates) and the process correspond through the Feynman-Kac formula
$$e^{t \Delta_D} f (x) = \mathbb{E}_x \left[ f(B_t) 1_{\{t < \zeta\}}\right], \quad x \in M, t > 0.$$
Here, $\mathbb{E}_x$ denotes the expectation with respect to the process  started at $x \in M$; we refer to \cite{FOT} for details. The so-constructed process $(B_t)_{t > 0}$ is called the {\em minimal Brownian motion} on $M$. If the underlying manifold is a bounded open subset of Euclidean space, minimal Brownian motion is Euclidean Brownian motion stopped upon hitting the boundary and the life-time  is given by its first exit time from the domain.

One fundamental question about minimal Brownian motion is about its {\em stochastic completeness}, i.e., whether its life-time is infinite. It follows from the Feynman-Kac formula that an infinite life-time is equivalent to 
$$\mathbb{P}_x(B_t \in M) = e^{t \Delta_D} 1 (x) = 1, \quad x \in M, t>0.$$
Due to the second equality involving the semigroup, in this case we also say that $(e^{t \Delta_D})_{t > 0}$ is {\em conservative} (the semigroup solution to the heat equation preserves the total amount of heat in the system).  Of course, on bounded open subsets of Euclidean space Brownian motion will eventually exit the domain so that it is always stochastically incomplete. Since bounded open domains in Euclidean space are not geodesically complete, one may wonder whether stochastic completeness is related to geodesic completeness but in \cite{Aze} a geodesically complete but stochastically incomplete manifold is constructed. It is a manifold with large negative sectional curvature, which forces Brownian motion to exit the manifold in finite time. 

This example brought up the need to investigate the interplay between the geometry of the manifold and stochastic completeness. Here we mention the pioneering works \cite{Gaf,Kas,Yau} that culminated in \cite{Gri2} with the insight that on a geodesically complete manifold stochastic completeness is related to the volume growth of balls. More precisely, geodesically complete manifolds whose volume of balls does not grow too fast are stochastically complete,  see also \cite{Dav,Gri1,Tak89} for related material. Since the discussed construction of Brownian motion through the Feynman-Kac formula uses the abstract machinery of Dirichlet forms, it does not come as a surprise that an optimal volume growth criterion for stochastic completeness was later also obtained for diffusion processes induced by strongly local regular Dirichlet forms on more general state spaces than manifolds, see \cite{Stu1}, and for jump processes on discrete spaces, see \cite{Fol,Hua}.

All proofs for the optimal volume growth test for stochastic completeness have in common that they are based on an analytic criterion  of Khasminskii \cite{Kas}. It says that stochastic completeness holds if and only if  all bounded smooth solutions to the initial value problem for the heat equation
$$\begin{cases}
 \partial_t u = \Delta u    \\
 u_0 =0
\end{cases} $$
equal zero.

From the PDE  viewpoint it is also natural to study this uniqueness property of the heat equation  with the Laplacian replaced by the Schrödinger operator $H = \Delta  - V$. Moreover, if $V$ is nonnegative, the self-adjoint realization $H_D = \Delta_D  - V$ with generalized Dirichlet boundary conditions   generates a Markovian semigroup $(e^{tH_D})_{t>0}$ and hence corresponds to a Markov process through the Feynman-Kac formula. Locally this process behaves like Brownian motion, but certain paths are stopped earlier  as the nonnegative potential drains heat from inside the system.

One may wonder whether in this case also  conservativeness of the semigroup generated by $H_D$ is related to uniqueness of bounded smooth solutions to the heat equation with respect to $H$. While the latter property may or may not be satisfied, if $V \neq 0$, the semigroup $(e^{tH_D})_{t>0}$ is never conservative. This follows from the observation that $(t,x) \mapsto e^{tH_D}1(x)$ is always a solution to the heat equation $\partial_t u = Hu$ while  constant functions do not solve this equation when $V \neq 0$. However, it was recently observed in \cite{KL} for discrete Schrödinger operators with nonnegative potentials (such operators are induced by infinite weighted graphs) that uniqueness of bounded solutions to the heat equation is equivalent to a generalized conservation property. On manifolds the {\em generalized conservation property} takes the form
$$e^{tH_D} 1 (x) + \int_0^t e^{sH_D} V (x) ds = 1 , \quad x \in M, t>0.$$
Roughly speaking, the additional quantity $\int_0^t e^{sH_D} V (x) ds$ corresponds to the probability that the associated process  is stopped up to time $t$ due to the presence of the potential.  Hence, the generalized conservation property says that heat is only lost inside 
the manifold due to the presence of the potential and not at ``infinity''.  

It is the main goal of this paper to extend Khasminskii's criterion for the generalized conservation property to Schrödinger operators on manifolds. This is achieved in Theorem~\ref{theorem:non conservative} and Theorem~\ref{theorem:conservative}. Besides the criterion involving uniqueness of bounded solutions to the heat equation  we also give a version for uniqueness of bounded solutions to the eigenvalue equation.  Moreover, we discuss several applications. 

We show that on stochastically complete manifolds the generalized conservation property holds for all nonnegative potentials. Since this result is perturbative in nature, we would like to stress that the generalized conservation property is not obtained by mere perturbation theory. More precisely,  there exists a stochastically incomplete manifold and a potential such that $H_D$ satisfies the generalized conservation property, see the discussion in Subsection~\ref{subsection:model manifolds} and Subsection~\ref{subsection:large potentials}. 

As a second application we prove that the generalized conservation property for $H_D$ is equivalent to the conservation property of $(1 + V)^{-1}\Delta_D$ considered on the $L^2$-space with respect to the changed measure $(1+ V) \vo$, see Theorem~\ref{theorem:criteria for generalized conservation}. In terms of associated processes the latter operator corresponds to a diffusion process on the manifold obtained from a time change of minimal Brownian motion. This observation is new, in particular it is not included in the aforementioned \cite{KL} for  discrete spaces. Since geometric criteria for conservativeness of general diffusions are rather well-understood, see \cite{Stu1}, it opens the way for geometric criteria  for the generalized conservation property. In Subsection~\ref{subsection:model manifolds} we pursue this strategy to characterize the generalized conservation property  through volume growth criteria  on model manifolds. Similar results on discrete spaces for weakly spherically symmetric graphs are contained in \cite{KLW}. Moreover, in Subsection~\ref{subsection:large potentials} we use the volume growth test for stochastic completeness from \cite{Stu1} to show that on every complete manifold there is a potential such that the semigroup generated by $H_D$ is conservative in the generalized sense. 

Minimal Brownian motion is defined with respect to the Laplacian with generalized Dirichlet boundary conditions. Other self-adjoint realizations of the Laplacian that generate Markovian semigroups can lead to other instances of Brownian motion according to the corresponding boundary conditions. However, it is known that if minimal Brownian motion is stochastically complete, there are no  self-adjoint realizations of the Laplacian that generate Markovian semigroups other than $\Delta_D$ and hence Brownian motion is unique, see e.g. \cite{GM,HKLMS, Kuw} for this result in various situations and generality. Analytically uniqueness of Brownian motion corresponds to the identity of the Sobolev spaces $W_0^1(M) = W^1(M)$ on the underlying manifold $M$. In Theorem~\ref{theorem:markov uniqueness} we extend the observation that stochastic completeness implies $W_0^1(M) = W^1(M)$ to the Schrödinger operator case, i.e., we prove that if $(e^{tH_D})_{t>0}$ satisfies the generalized conservation property, then $W^1_0(M,1 + V) = W^1(M,1 + V)$, where the latter denote weighted Sobolev spaces on the manifold, see Subsection~\ref{subsection:sobolev spaces} below.  

Concerning our presentation of the subject we aim at two different audiences, namely people working on PDEs on manifolds and people acquainted with Dirichlet forms and stochastics. In order to reach both, compromises are necessary at several places. We finish this introduction by  trying to explain  the compromises and discuss where the strengths and the limits of our methods lay.

   The most prominent class of spaces for global analysis are certainly smooth Riemannian manifolds. Therefore, we chose to focus on them, consider only smooth nonnegative potentials  and formulate the extension of Khasminskii's criterion and applications in the smooth category.  On a technical level this is possible because we can apply parabolic and elliptic local regularity theory, which we explain in detail.  
   
   As should  be clear from the preceding discussion, the concept of the generalized conservation property also makes sense for other Markovian semigroups and an extension of Khasminskii's criterion is certainly of interest for  Markovian semigroups associated with  regular Dirichlet forms. The methods that we use in this text are strong and abstract enough to treat this case.  We comment at the corresponding places on  necessary changes and believe that it should be no problem for experts on Dirichlet forms to fill the gaps.  However, since there is no local parabolic and elliptic regularity theory for general Dirichlet forms,  smooth strong solutions need to be replaced by weak solutions in the local form domain. As mentioned above, for operators on discrete spaces Khasminskii's criterion for the generalized conservation property is contained in \cite{KL} and on a very abstract level similar results (only the elliptic but not the parabolic part) can be found in \cite{Schmi}. In contrast to these two texts we would like to point out that we work with parabolic maximum principles instead of elliptic maximum principles. 
   
   We mentioned at various places in the introduction that the generalized conservation property can be understood in terms of associated Markov processes. However, in order to keep the text reasonably short and accessible,  our  definitions, methods and proofs are purely analytical. We hint on the stochastic relevance of certain formulas and theorems but leave details to the reader.

   {\bf Acknowledgements:} The authors would like to thank Matthias Keller and Daniel Lenz for introducing them to the generalized conservation property on discrete spaces. Moreover, they are grateful to Masatoshi Fukushima for interesting discussions on the stochastic background of the generalized conservation property. A substantial part of this work was done while M.S.\ was visiting GSIS at Tohoku University Sendai and the Department of Mathematics at Hokkaido University Sapporo and while J.M.\ was visiting Fakultät für Mathematik und Informatik at Friedrich-Schiller-Universität Jena.
They expresses their  warmest thanks  to  these  institutions.  Furthermore, they acknowledge the financial support of JSPS ``Program for Advancing Strategic International Networks to Accelerate the Circulation of Talented Researchers", and J.M.\ acknowledges the financial support of JSPS No.16KT0129.


%

\section{Preliminaries}
In this section we introduce the notation and the  objects that are used throughout the text. We basically follow \cite{Gri}, to which we  refer the reader for formal definitions. For a background on Dirichlet form theory see \cite{FOT,MR}. All functions in this text are real-valued.

\subsection{Distributions and Schrödinger operators}

Let $M=(M,g,\mu)$ be a smooth connected weighted Riemannian manifold without boundary. 
We assume that the measure $\mu$ has a smooth and strictly positive density  
$\Psi$ on $M$ against the Riemannian measure $\vo$.

For $1 \leq p < \infty$ we denote by $L^p(M,\mu)$ the Lebesgue space of $p$-integrable functions with norm $\|\cdot\|_p$. If $p = 2$, it is a Hilbert space with inner product
$$\as{f,g}_{\mu} = \int_M fg d \mu.$$
Since $\mu$ is assumed to have a smooth and strictly positive density with respect to $\vo$, the space of essentially bounded functions is independent of the particular choice of $\mu$ and we denote it by $L^\infty(M)$ and the corresponding norm by $\|\cdot\|_\infty$.  The same holds true for the local Lebesgue spaces; for $1 \leq p \leq \infty$ we denote them by $L^p_{\rm loc}(M)$. Similarly, we write $L^0(M)$ for the space of all real-valued $\mu$-a.e. defined measurable functions and $L^+(M)$ for the cone of all $[0,\infty]$-valued $\mu$-a.e. defined measurable functions. Note that functions in $L^+(M)$ may take the value $\infty$ on a set of positive measure.  For two functions $f,g \in L^0(M)$ we let $f \wedge g = \min\{f,g\}$ and $f \vee g = \max \{f,g\}$, where the maximum and minimum are taken pointwise $\mu$-a.e. Moreover, we write $f_+ = f \vee 0$ and $f_- = (-f) \vee 0$ for the positive and negative part of $f$, respectively.

The space of continuous functions on $M$ is denoted by $C(M)$ and equipped with the topology of uniform convergence on compact sets. We write $C_b(M)$ for the subspace of bounded continuous functions. Moreover, $C^\infty(M)$ is the space of smooth functions on $M$ and and $C_c^\infty(M)$ is the space of smooth functions of compact support.

The space of   {\em test functions} $\dm = C_c^\infty(M)$ is equipped with the usual locally convex topology, cf. \cite[Chapter~4.1]{Gri}. Its  continuous dual, the space of {\em distributions}, is denoted by $\dmp$ and we write $\as{\cdot,\cdot}$ for the  dual pairing between $\dmp$ and $\dm$. A sequence of distributions $(u_n)$ converges to a distribution $u$ if $\as{u_n,\varphi} \to \as{u,\varphi}$ for each $\varphi \in \dm$. We identify functions $f \in L^1_{\rm loc}(M)$ with the distribution
$$\dm \to \R,\, \varphi \mapsto \int_M f  \varphi d \vo.$$
 The so-obtained map $L^1_{\rm loc}(M) \to \dmp$ is injective and continuous. For later  it is important to note that we use the measure $\vo$ and not the measure $\mu$ for this identification. In particular, for $u \in L^2(M,\mu)$ (when viewed as distribution)  and $\varphi \in \dm$ we have
$$\as{u,\Psi \varphi} = \as{u,\varphi}_\mu,$$
where we recall that $\Psi$ is the smooth density of $\mu$ against $\vo$.

By $\Dm = \Dm_\mu$ we denote the {\em weighted Laplacian} of $(M,g,\mu)$ acting on $\dmp$; namely,
$$
\Dm u= \frac{1}{\Psi} \mbox{div}(\Psi \nabla u).
$$
Here, $\nabla$ and $\mbox{div}$ are the distributional versions of the gradient and the divergence operator, respectively. The Laplacian leaves $\dm$ invariant and satisfies
$$\as{\Dm u, \Psi \varphi} = \as{u, \Psi \Dm \varphi}, \quad u \in \dmp, \varphi \in \dm.$$
This identity and our choice of the inclusion $\Lol \hookrightarrow \dmp$ imply that the restriction of $\Dm$ to $\dm$ is a symmetric operator on $L^2(M,\mu)$.

To a  smooth function $V:M \to \R$  and a strictly positive smooth function $\rho:M \to (0,\infty)$  we associate the {\em Schr\"odinger operator} $\cL: \mathcal{D}'(M) \to \mathcal{D}'(M)$ via
$$
\cL= \rho^{-1}\Delta - \rho^{-1}V.
$$
Its restriction to $\dm$ is a symmetric operator on $L^2(M,\rho\mu)$. The function $V$ is called the {\em potential} and the function $\rho$ is called the {\em density} of $\cL$. Even though not strictly necessary, below we always assume the following.
\begin{assumption}
 The potential $V$ is nonnegative. 
 \end{assumption}
At various places we discuss which additional challenges a potential without a fixed sign poses and how these difficulties could possibly be overcome.

A distribution $u \in \mathcal{D}'((0,\infty) \times M)$ is called a (weak) {\em solution to the heat equation} (with respect to $\cL$) if 
$$\partial_t u  = \cL u$$
holds in the sense of $\mathcal{D}'((0,\infty) \times M)$. Due to local regularity theory, see Appendix~\ref{section:local regularity theory}, every weak solution to the heat equation automatically belongs to $C^\infty((0,\infty) \times M)$.  We usually write the time variable as a lower index, i.e., we let $u_t = u(t,\cdot)$ if $u$ is a function depending on time and space.  We say that a solution to the heat equation $u \in C^\infty((0,\infty) \times M)$ has {\em initial value $f \in C(M)$} if it continuously extends to $[0,\infty) \times M$ (in which case we write $u \in C([0,\infty)\times M)$) and satisfies $u_0 = f$. Note that this is equivalent to the local uniform convergence $u_t \to f$, as $t \to 0+$.



\subsection{Sobolev spaces and self-adjoint realizations} \label{subsection:sobolev spaces}
We define the {\em first order weighted Sobolev space} $W^1(M,\rho) = W^1_\mu(M,\rho)$ with respect to the density $\rho$ by
$$W^1(M,\rho) = \{f \in L^2(M,\rho \mu) \mid \nabla f \in \vec L^2(M,\mu)\}$$
and equip it with the norm 
$$\|f\|_{W_1} = \left( \int_M |\nabla f|^2 d\mu + \int_M f^2 \rho d\mu\right)^{1/2}.$$
The closure of $C_c^\infty(M)$ in this space is denoted by $W^1_0(M,\rho) = W_{0,\mu}^1(M,\rho).$ If $\rho = 1$, we simply write $W^1(M)$ and $W_0^1(M)$ for $W^1(M,1)$ respectively $W^1_0(M,1)$.

In this paper we are concerned with properties of the self-adjoint realization of $\cL$ that has generalized Dirichlet boundary conditions, which we introduce next. We let 
$$D(\Qr) = W_0^1(M,V+\rho)$$
the domain of the quadratic form $\Qr:D(\Qr) \times D(\Qr) \to \R$ on which it acts by
$$ Q_{\rho, V}(f,g) = \int_M \as{\nabla f,\nabla g} d\mu + \int_M fg V d\mu.$$
It is a  {\em Dirichlet form on $L^2(M,\rho \mu)$}, i.e., it is a nonnegative closed quadratic form and for any normal contraction $C$ (a $1$-Lipschitz function $C:\R \to \R$ with $C(0) = 0$) and all $f \in D(Q_{\rho, V})$ we have $C \circ f \in D(Q_{\rho,V})$ and $Q_{\rho, V}(C \circ f) \leq Q_{\rho,V}(f)$. Here and in what follows we drop one argument when evaluation the diagonal of bilinear forms. This means that we use the convention $Q_{\rho,V}(f) = Q_{\rho,V}(f,f)$ for $f \in D(Q_{\rho,V})$.

It follows from the distributional definition of $\cL$ that for all $f \in D(\Qr)$ with $\cL f \in L^2(M,\rho \mu)$ and $\varphi \in \dm$ we have 
$$\Qr(f,\varphi) = \as{-\Delta f + V f,\Psi \varphi}  = \as{-\cL f,\varphi}_{\rho \mu}.$$
We denote by $L_{\rho,V}$   the nonpositive self-adjoint operator on $L^2(M,\rho \mu)$ that is associated with the closed form $Q_{\rho, V}$. By the above integration by parts formula it is a  restriction of $\cL$ to the domain
$$D(L_{\rho,V})=\{f \in D(Q_{\rho,V}) \mid \mathcal{L}_{\rho,V} f \in L^2(M, \rho\mu)\},$$
i.e., it acts by
$$L_{\rho,V} f = \cL f, \quad f \in D(L_{\rho,V}).$$
\begin{remark}
 The operator $L_{\rho,V}$ can be thought of having generalized Dirichlet boundary conditions or Dirichlet boundary conditions ''at infinity``. If $V = 0$ and $\rho = 1$, then $L_{1,0}$ is a  self-adjoint realization of the weighted Laplacian.
 In our text it is important to allow nonvanishing potentials as well as some flexibility on the measure.
\end{remark} 
%
%
%
%
%
%
%
%
Since $\Qr$ is a Dirichlet form, the  associated semigroup $\Tr = e^{t\Lr}$, $t>0$,  and resolvent $ \Gr = (\alpha-\Lr)^{-1}$, $\alpha > 0$, are Markovian, i.e., for each $e \in L^2(M,\rho \mu)$ with $0 \leq e \leq 1$ they satisfy
$$0 \leq \Tr e \leq 1 \text{ and } 0 \leq \alpha \Gr e \leq 1.$$

Let $T:L^2(M,\rho\mu) \to L^2(M,\rho\mu)$  be a positivity preserving linear operator, i.e., an operator for which $f \geq 0$ implies $Tf \geq 0$. Then $T$ can be extended to an operator $\tilde T:L^+(M) \to L^+(M)$ by letting
$$\tilde Tf = \lim_{n \to \infty} Tf_n,$$
where $(f_n)$ is an increasing sequence of nonnegative functions in $L^2(M,\rho\mu)$ with $f_n \to  f$ $\mu$-a.e.  It is proven in \cite{Kaj} that this is well defined and a linear operator on the cone $L^+(M)$. In particular, if we let
$${\rm dom}(T) = \{f \in L^0(M) \mid \tilde T|f| < \infty\, \mu\text{-a.e.}\},$$
then $\tilde T: {\rm dom}(T) \to L^0(m)$ defined by $\tilde T f:= \tilde T{f_+} - \tilde T{f_-}$ is a linear operator that extends $T$. In what follows we shall abuse notation and write $T$ for the given operator and its extension to ${\rm dom}(T)$. We write ${\rm dom}$ for the domain of this extension because we reserve the capital $D$  for the $L^2$-domain of an unbounded operator or quadratic form.

The semigroup $(T^{\rho,V}_t)$ and the resolvent $(G_\alpha^{\rho,V})$ are positivity preserving.  It follows from their Markov property that $L^\infty(M) \subseteq {\rm dom}(T^{\rho,V}_t),{\rm dom}(G_\alpha^{\rho,V})$  and that $T^{\rho,V}_t 1 \leq 1$  and $\alpha G_\alpha^{\rho,V}1 \leq 1$. Therefore,  they act as contractions on the space $L^\infty(M)$. These extensions are weak-$*$-continuous on $L^\infty(M)$ in the parameters $t$ respectively $\alpha$.   Since the semigroup and the resolvent are self-adjoint on $L^2(M)$ and Markovian, by duality they can be extended to strongly continuous semigroups on $L^1(M,\rho\mu)$.

We recall the following basic definition concerning  the extension to $L^\infty(M)$.

\begin{definition}[Conservativeness and stochastic completeness]
The Dirichlet form $\Qr$ is called {\em conservative} if   $T^{\rho,V}_t 1 = 1$ for all $t > 0$.  The weighted manifold $(M,g,\mu)$ is called {\em stochastically complete} if the Dirichlet form $Q_{1,0}$ on $L^2(M,\mu)$ is conservative.
\end{definition}

\section{A generalized conservation property - definition and characterizations}

It is well a well known fact in Dirichlet form theory that Dirichlet forms with nonvanishing killing, in our case a nonvanishing potential $V$, cannot be conservative. In this section,  we introduce a generalized conservation criterion invoking the potential. It is inspired by the corresponding definition for infinite weighted graphs that was given in \cite{KL}.   We prove that Khasminskii's criterion \cite{Kas} for  stochastic completeness (conservativeness for the form with vanishing potential), which characterizes stochastic completeness in terms of unique solvability of the heat equation in $L^\infty$, remains valid for the generalized conservation property with the Laplacian replaced by the Schrödinger operator. This can be seen as the main result of our paper.

In what follows we let $\hat V = V / \rho$. For $t > 0$ we define $H_t \in L^+(M)$ by
$$H_t  = T^{\rho,V}_t 1 + \int_0^t T_s^{\rho,V} \hat V ds.$$
 Here, $\int_0^t T_s^{\rho,V} V ds$ is the $L^+(M)$-extension of the positivity preserving operator
 $$A_t:L^2(M,\rho\mu) \to L^2(M,\rho\mu),\, f \mapsto \int_0^t T_s^{\rho,V} f ds,$$ 
 applied to the nonnegative function $\hat V$. Moreover,  for $\alpha > 0$ we define $N_\alpha \in L^+(M)$ by
$$N_\alpha = \alpha G_\alpha^{\rho,V} 1 + G_\alpha^{\rho,V} \hat V.$$
 The following theorem  is  the main technical insight of this paper. It discusses properties of the functions $H_t$ and $N_\alpha$. In particular, it shows that $\hat V \in {\rm dom}(G_\alpha^{\rho,V})$ and $\hat V \in {\rm dom}(A_t)$  so that the functions $N_\alpha$ and $H_t$ are finite. We postpone its proof to Section~\ref{section:proof of the main theorems}.

\begin{theorem}\label{theorem:main result}
\begin{enumerate}[(a)]
                 \item For each $t > 0$ we have $H_t \in C^\infty(M)$. The function $H:(0,\infty) \times M \to \R, (t,x) \mapsto H_t(x)$ satisfies $0 \leq H \leq 1$ and belongs to $C^\infty((0,\infty) \times M) \cap C([0,\infty) \times M)$. It solves the  equation
                  $$
                                \begin{cases}
                                 (\partial_t - \cL) H = V /\rho \text{ on }(0,\infty) \times M,\\
                                 H_0 = 1.
                                \end{cases}
                 $$
                 Furthermore, $1 - H$ is the largest function $u \in  C^\infty((0,\infty) \times M) \cap C([0,\infty)\times M)$ which satisfies $u \leq 1$ and
                  $$
                 \begin{cases}
                (\partial_t - \cL) u \leq 0 \mbox{ on }(0,\infty) \times M,\\
                 u_0 = 0.
                                 \end{cases}
                 $$
                 \item For each $\alpha>0$ the function $N_\alpha$ belongs to $C^\infty(M)$, satisfies $0\leq N_\alpha \leq 1$ and solves 
                     $$(\alpha - \cL)N_\alpha = \alpha 1 +  V / \rho. 
                       $$
                     Furthermore, $1 - N_\alpha$ is the largest function $g \in  C^\infty(M)$ which satisfies $g \leq 1$  and   
                      $$
                     (\alpha-\cL) g \leq 0.
                     $$
                  \item For every $\alpha > 0$ and $x \in M$ we have
                  $$\int_0^\infty \alpha e^{-t\alpha} H_t(x) dt = N_\alpha(x).$$
\item The following dichotomy holds. Either $H_t(x) = 1 = N_\alpha (x)$ for all $\alpha,t > 0$ and all $x \in M$ or $H_t(x) < 1$ and $N_\alpha(x) < 1$ for all $\alpha,t > 0$ and all $x \in M$.
 %
                \end{enumerate}
\end{theorem}
\begin{remark}\label{remark:measurability H}
As is standard, (a) and (b) mean that $H_t$ respectively $N_\alpha$ have smooth versions. In what follows we always work with those versions. Only assertion (d) uses that $M$ is assumed to be connected.
%
\end{remark}

With these properties of $H$ and $N_\alpha$ at hand, we can now introduce the generalized conservation property, the main concept of this paper.

 \begin{definition}[Conservativeness in the generalized sense]
 The Dirichlet form $Q_{\rho,V}$ (respectively the semigroup $(\Tr)_{t > 0}$) is called {\em conservative in the generalized sense} if $H_t = 1$ for all $t>0$.
\end{definition}
 \begin{remark}
  For $V = 0$ conservativeness and  conservativeness in the generalized sense of $\Tr$ coincide. The quantity $H_t$ has an interpretation in terms of the heat flow. If we study semigroup solutions to the heat equation (with respect to $\cL$) with initial value $1$, which corresponds to a uniform initial heat distribution, then $T^{\rho,V}_t 1$ is the density of the total amount of heat in the system at time $t$. It  can decrease over time for two reasons.  Either heat is transported to the boundary of $M$ (which can be thought of laying at infinity) or heat is lost inside of $M$ due to the presence of a potential, which drains heat from the system. The amount of heat lost by the latter effect can (heuristically) be computed as 
  $$ \int_0^t T^{\rho,V}_s \hat V ds,$$
 cf. the discussion in \cite[Section~8]{KL2}, which treats the same phenomenon for Dirichlet forms on graphs. Hence, the amount of heat transported to the boundary at infinity is $1 - H_t$, so that   $H_t = 1$ for all $t >0$ if and only if no heat is transported to the boundary at infinity. This generalized conservation property was first introduced in \cite{KL} for Dirichlet forms on graphs. Due to the previously described interpretation for $H_t = 1$,  Dirichlet forms which are conservative in the generalized sense are called stochastically complete at infinity in \cite{KL}. In our terminology stochastic completeness is a property of the weighted manifold as a geometric object viz. the conservativeness of its canonical Dirichlet form. In contrast, we think of conservativeness (in the generalized sense) as a property of abstract Markovian semigroups, where $V$ and $\rho$ appear as an additional non-geometric input. This is why we do not use the term stochastic completeness at infinity.
 \end{remark}

 %
%

 %

The main results of this paper are the following characterizations of the generalized conservation criterion. As mentioned above, they are extensions of the classical characterization of stochastic completeness by Khasminskii \cite{Kas}, which treats the case $\rho = \Psi = 1$ and $V = 0$.

\begin{theorem} \label{theorem:non conservative}
 The following assertions are equivalent.
\begin{itemize}
 \item[(i)]  The function $1-H$ is nontrivial. 
 \item[(ii)] For  some/any $\alpha > 0$ the function $1-N_\alpha$ is nontrivial. 
 \item[(iii)] For some/any $\alpha > 0$ there exists a nontrivial bounded $g \in C^\infty(M)$ with 
 		$$(\alpha - \cL)g = 0.$$
 \item[(iv)]  For some/any $\alpha > 0$ there exists a nontrivial  nonnegative bounded $g \in C^\infty(M)$ with 
 		$$(\alpha - \cL)g \leq 0.$$
 \item[(v)] There exists a nontrivial  $u \in C^\infty((0,\infty) \times M) \cap C_b([0,\infty) \times M)$ that satisfies
   $$
                 \begin{cases}
                (\partial_t - \cL) u = 0 \mbox{ on } (0,\infty) \times M,\\
                     u_0 = 0.
                                 \end{cases}
             $$
 \item[(vi)] There exists a nontrivial nonnegative   $u \in C^\infty((0,\infty) \times M) \cap C_b([0,\infty) \times M)$ that satisfies 
  $$
                 \begin{cases}
                (\partial_t - \cL) u\leq 0 \mbox{ on } (0,\infty) \times M,\\
                     u_0 = 0.
                                 \end{cases}
             $$
 \end{itemize}
\end{theorem}

 \begin{proof}  (i) $\Leftrightarrow$ (ii):  We use the identity
 $$ \int_0^\infty e^{-t\alpha} H_t\, dt = G_\alpha 1 + \frac{1}{\alpha} G_\alpha \hat V. $$
 It shows that $H = 1$ implies $N_\alpha = 1$ for any $\alpha > 0$. Hence, the nontriviality of $1- N_\alpha$ for one $\alpha> 0$ shows the nontriviality of $1-H$.
 
 For the other implication note that by  Theorem~\ref{theorem:main result} the function $H$ is smooth and satisfies $H \leq 1$. This implies that if $H \neq 1$, then there exist $0 < s < t$ and $\Omega \subseteq M$ open such that $H \leq C < 1$ on  $(s,t) \times \Omega$. By the previous equation we obtain that $N_\alpha < 1$ on $\Omega$ for any $\alpha > 0$. Thus, the nontriviality of $1-H$ implies that for any $\alpha >0$ the function $1-N_\alpha$ is nontrivial. 
 
 (i) $\Rightarrow$ (v) and (i) $\Rightarrow$ (vi): By  Theorem~\ref{theorem:main result} the function $1-H$ has the desired properties.
 
 (v)  $\Rightarrow$ (i): Let $u \in C^\infty((0,\infty) \times M) \cap C_b([0,\infty) \times M)$ nontrivial with
   $$
                 \begin{cases}
                (\partial_t - \cL) u = 0 &\mbox{on } (0,\infty) \times M,\\
                     u_0 = 0.
                                 \end{cases}
             $$
             Without loss of generality we can assume $|u| \leq 1$. Theorem~\ref{theorem:main result}~(a) applied to $u$ and $-u$ yields $u \leq 1- H$ and $-u \leq 1-H$ so that $|u| \leq 1-H$. Since $u$ is nontrivial, this implies that $1-H$ is nontrivial.
             
  (vi)$\Rightarrow$ (i): Let   $u \in C^\infty((0,\infty) \times M) \cap C_b([0,\infty) \times M)$ nontrivial and nonnegative  with
  $$
                 \begin{cases}
                (\partial_t - \cL) u\leq 0 &\mbox{on } (0,\infty) \times M,\\
                     u_0 = 0.
                                 \end{cases}
             $$
      Without loss of generality we can assume $u \leq 1$. Theorem~\ref{theorem:main result}~(a) shows $u\leq 1-H$. Since $u$ is nonnegative and nontrivial, this implies that $1-H$ is nontrivial.
      
      (ii) $\Leftrightarrow$ (iii) $\Leftrightarrow$ (iv): This can be proven along the same lines as the other implications using Theorem~\ref{theorem:main result}~(b) instead of Theorem~\ref{theorem:main result}~(a).
\end{proof}

For the sake of completeness we also mention the following characterization of the generalized conservation property. It is just the negation of the previous theorem. Recall that $\hat V = V / \rho$.
 
\begin{theorem} \label{theorem:conservative}
 The following assertions are equivalent.
 \begin{enumerate}[(i)]
  \item The Dirichlet form $\Qr$ on $L^2(M,\rho \mu)$ is conservative in the generalized sense, i.e., for all $t > 0$ we have 
  $$1 = T^{\rho,V}_t 1 + \int_0^t T^{\rho,V}_s \hat V\, ds.$$  
  \item For all $\alpha >0$ we have
  $$1 = \alpha G^{\rho,V}_\alpha 1 + G^{\rho,V}_\alpha \hat V.$$  
  \item For one/any $\alpha > 0$  any bounded $g \in C^\infty(M) $ with $(\alpha- \cL) g = 0$ satisfies $g = 0$.
  \item For one/any $\alpha > 0$  any nonnegative bounded $g \in C^\infty(M)$  with $(\alpha - \cL) g \leq 0$ satisfies $g = 0$. 
  \item Any solution $u \in C^\infty((0,\infty) \times M) \cap C_b([0,\infty) \times M)$ to the equation 
  $$
                 \begin{cases}
                (\partial_t - \cL) u = 0 \mbox{ on } (0,\infty) \times M,\\
                     u_0 = 0,
                                 \end{cases}
             $$
   satisfies $u = 0$.
  \item  Any nonnegative $u \in C^\infty((0,\infty) \times M) \cap C_b([0,\infty) \times M)$ that satisfies the differential inequality
 $$
                 \begin{cases}
                (\partial_t - \cL) u \leq 0 \mbox{ on } (0,\infty) \times M,\\
                     u_0 = 0,
                                 \end{cases}
             $$
             satisfies $u=0$.
 \end{enumerate}
\end{theorem}

\begin{remark}
\begin{enumerate}
 \item  Theorem~\ref{theorem:non conservative} can be seen as a generalization of \cite[Theorem~1]{KL} on graphs, which does not include our assertion (ii), to the manifold case. An abstract version of the equivalence of (ii), (iii) and (iv) in Theorem~\ref{theorem:conservative} is given by \cite[Theorem~4.68]{Schmi}, which treats all Dirichlet forms. There however only weak solutions are considered and the equivalence of (i), (v) and (vi) is missing. In contrast to the proofs in \cite{KL,Schmi}, which use elliptic maximum principles, our proof relies on a parabolic maximum principle.
 \item In this text we chose to work in the smooth category because   the input data $(M,g,\mu)$ and $\rho,V$ are assumed to be smooth. If they were not, we could not use elliptic and parabolic regularity theory as we do here. However, the previous theorem would still hold true with essentially the same proof but with $C^\infty$-solutions replaced by  weak solutions in the local form domain. 
\end{enumerate}
 \end{remark}

As a consequence to the previous characterizations  we obtain that conservativeness with vanishing potential implies conservativeness in the generalized sense for all nonnegative potentials. Moreover, conservativeness in the generalized sense is the same as conservativeness of a Dirichlet form with changed measure.  In terms of associated stochastic processes this corresponds to a time change.

 \begin{theorem}   \label{theorem:criteria for generalized conservation}
 \begin{enumerate}[(a)]
                  \item If $Q_{\rho,0}$ the is conservative, then $\Qr$ is conservative in the generalized sense.  In particular, if $(M,g,\mu)$ is stochastically complete, then $Q_{1,V}$ is conservative in the generalized sense.
                  \item The Dirichlet form $\Qr$ on $L^2(M,\rho\mu)$ is conservative in the generalized sense if and only if the Dirichlet form $Q_{\rho + V,0}$ on $L^2(M,(\rho + V) \mu)$ is conservative.

 \end{enumerate}
 
 \end{theorem}
 
 \begin{proof}
 (a): Let $\alpha > 0$ and let $u \in C^\infty(M)$ nonnegative and bounded with $$(\alpha - \rho^{-1} \Delta + \rho^{-1}V) u =  (\alpha - \cL) u \leq 0.$$ According to  Theorem~\ref{theorem:conservative} it suffices to prove $u = 0$. Since $\rho^{-1}V$ and $u$ are nonnegative, the assumptions on $u$ imply $(\alpha - \mathcal{L}_{\rho,0}) u = (\alpha - \rho^{-1} \Delta)u \leq 0$.  The conservativeness of $Q_{\rho,0}$ yields $u = 0$ by Theorem~\ref{theorem:conservative}. 
 
  (b): Let $u \in C^\infty(M)$ be a solution to the equation $(1 - \cL) u = 0$. By definition of $\cL$  and since $\rho > 0$, this is equivalent to $\rho u + V u = \Delta u$. This in turn holds if and only if $\mathcal{L}_{\rho + V,0} u = (\rho + V)^{-1} \Delta u = u$.  Therefore, $(1 - \cL) u = 0$ has  a nontrivial bounded solution in $C^\infty(M)$ if and only if $(1 -  \mathcal{L}_{\rho + V,0})u = 0$  has  a nontrivial bounded solution in $C^\infty(M)$. With this at hand, the claim follows from Theorem~\ref{theorem:non conservative}.
 \end{proof}
 
\begin{remark}
 For graphs assertion (a) is contained in \cite{KL}. Assertion (b) seems to be a new observation. Since conservativeness of Dirichlet forms is quite well-understood, it opens the way to studying the generalized conservation property. In  Subsection~\ref{subsection:model manifolds} we use this strategy to provide a characterization of the generalized conservation property on model manifolds in terms of   volume growth.  In Subsection~\ref{subsection:large potentials} we employ known volume growth criteria for the conservativeness of $Q_{\rho + V,0}$ to obtain that on any complete manifold there is a potential such that  $Q_{1+ V,0}$ is consevative in the generalized sense.
\end{remark}

\section{Maximum principles}\label{section:maximum principles}

In this section we discuss a parabolic and an elliptic maximum principle for the operator $\cL$. Both are used in the proof of the main results. The proofs that we give apply to more general situations and so they may be of independent interest.

We let $\Qt$ the Dirichlet form on $L^2(M,\rho \mu)$ with domain $D(\Qt) = W^1(M,\rho + V)$, on which it acts by
$$\Qt(f,g) = \int_M \as{\nabla f,\nabla g} d \mu + \int_M Vfg d\mu.$$
With this notation we have $\|f\|_{W^1}^2 = \Qt(f)  + \|f\|_2^2$ for $f \in W^1(M,\rho+V)$. It follows from   the definition of the distributional operator $\cL$ that  
$$\Qt(f,\varphi) = \as{-\cL f,\varphi}_{\rho \mu}, \quad f \in D(\Qt),\varphi \in C^\infty_c(M).$$
In particular, the associated self-adjoint operator is a restriction of $\cL$. The following observation lies at the heart of all maximum principles in this section.
\begin{lemma}\label{lemma:order ideal}
The space $W^1_0(M,\rho + V)$ is an order ideal in $W^1(M,\rho + V)$, i.e., for $f \in W^1_0(M,\rho + V),$ $g \in  W^1(M,\rho + V)$ the inequality $|g| \leq |f|$ implies $g \in W^1_0(M,\rho + V)$.
\end{lemma}
\begin{proof}
 Let $f,g$ as stated. Since $g_+,g_-  \in  W^1(M,\rho + V)$, $|f| \in W^1_0(M,\rho + V)$ (here we use that $\Qt$ and $Q_{\rho,V}$ are Dirichlet forms and $x \mapsto x \vee 0$, $y \mapsto |y|$ are normal contractions) and $|g_+|,|g_-| \leq |f|$, it suffices to consider the case $0 \leq g \leq f$. By definition of $W^1_0(M,\rho + V)$ there exists a sequence $(\varphi_n)$ in $C_c^\infty(M)$ with $\|f - \varphi_n\|_{W^1} \to 0$. Consider the functions 
 $$g_n :=  g \wedge |\varphi_n| = \frac{1}{2} ( g + |\varphi_n| - |g - |\varphi_n||).$$
 Since $0\leq  g \leq f$, the $(g_n)$ converge to $g$ in $L^2(M,\rho \mu)$. Moreover, $\Qt$ is a Dirichlet form and $\R \to \R, x\mapsto |x|$ is a normal contraction. Therefore, 
 \begin{align*}
   \Qt(g_n)^{1/2} &= \frac{1}{2} \Qt( g + |\varphi_n| - |g - |\varphi_n||)^{1/2} \\
   &\leq  \frac{1}{2} \Qt(g + |\varphi_n|)^{1/2} +   \frac{1}{2}\Qt(|g - |\varphi_n||)^{1/2} \\
   &\leq \Qt(g)^{1/2} + \Qt(\varphi_n)^{1/2}.
 \end{align*}
 This implies that $(g_n)$ is a bounded sequence in $W^1(M,\rho + V)$. By construction the $g_n$ have compact support. It follows with the same arguments as in \cite[Lemma~5.5]{Gri} that functions in $W^1(M,\rho + V)$ with compact support belong to  $W_0^1(M,\rho + V)$  (only the case $\rho = 1,V= 0$ is considered in \cite{Gri} but the argument is more general).  Hence $(g_n)$ is a bounded sequence in the Hilbert space $(W_0^1(M,\rho + V),\|\cdot\|_{W^1})$. The Banach-Saks theorem implies that it has a subsequence $(g_{n_k})$ such that its sequence of Cesàro means $\tilde g_N = N^{-1} \sum_{k = 1}^N g_{n_k}$ converges   to some $h \in W_0^1(M,\rho + V)$ with respect to $\|\cdot\|_{W^1}$. In particular, $\tilde{g}_N \to h$ in $L^2(M,\rho \mu)$. However, since $g_n \to g$ in $L^2(M,\rho \mu)$, we also have $\tilde g_N \to g$ in $L^2(M,\rho \mu)$, such that $g = h \in W_0^1(M,\rho + V)$. This proves the claim.
\end{proof}
\begin{remark}
The previous lemma says that the domain of the Dirichlet form $Q_{\rho,V}$ is an order ideal the domain of the Dirichlet form $\Qt$. According to \cite[Lemma~2.2]{Schmi2} this is equivalent to $D(\Qr) \cap L^\infty(M)$ being an algebraic ideal in $D(\Qt) \cap L^\infty(M)$. Form extensions of $\Qr$ with this property are called Silverstein extensions in the literature; hence $\Qt$ is a Silverstein extension of $\Qr$.
\end{remark}

%

We say that a function $f \in  W^1(M,\rho+V)$ satisfies 
$$f \geq 0 \mod W^1_0(M,\rho+V)$$
if there exists $g \in W^1_0(M,\rho+V)$ such that $f + g \geq 0$. For our purposes this is an adequate form of saying that $f$ is nonnegative on 'the boundary' of $M$. The following lemma characterizes when this inequality holds. It can be proven along the same lines as \cite[Lemma~5.12]{Gri} but we give an alternative proof that is based on the fact that  $W^1_0(M,\rho + V)$ is an order ideal in $W^1(M,\rho+V)$.
\begin{lemma} \label{lemma:boundary inequality}
A function   $f \in  W^1(M,\rho+V)$ satisfies $f \geq 0 \mod W^1_0(M,\rho+V)$ if and only if $f_- \in  W^1_0(M,\rho+V)$.
\end{lemma}
\begin{proof}
It follows  from the definition that $f_- \in  W^1_0(M,\rho+V)$ implies 
$$f \geq 0 \mod  W^1_0(M,\rho+V).$$  
Now assume that there exists some $g \in  W^1_0(M,\rho+V)$ such that $f + g \geq 0$.  We  have $|g| \in  W^1_0(M,\rho+V)$ and so we can assume $g \geq 0$.  This implies $0 \leq f_- \leq g$. Since $f_- \in W^1(M, \rho + V)$, the claim follows from Lemma~\ref{lemma:order ideal}.
\end{proof}

The following maximum principle is an extension of \cite[Theorem~5.16]{Gri} to the case when $V \neq 0$.
\begin{theorem}[Parabolic maximum principle] \label{theorem:parabolic maximum principle}
 Let $0 < T \leq \infty$ and let  $v:(0,T) \to  W^1(M,\rho + V)$ be a path with the following properties.
 \begin{itemize}
  \item $\partial_t v$ exists as a strong limit in $L^2(M,\rho\mu)$,
  \item $v(t)_- \to 0$ in $L^2(M,\rho\mu)$, as $t \to 0+$,
  \item for every $0 < t < T$, $v(t) \geq 0 \mod W_0^1(M,\rho + V)$,
  \item for every $0< s < T$, $(\partial_t v)(s) \geq \cL (v(s))$ in the sense of $\dmp$. 
 \end{itemize}
Then $v \geq 0$ on $(0,T) \times M$.
\end{theorem}
\begin{proof}
   According to Lemma~\ref{lemma:boundary inequality} the 'boundary condition' $v \geq 0 \mod W_0^1(M,\rho + V)$ implies $v_- \in  W_0^1(M,\rho + V)$. Thus, (for every fixed time) we can choose a sequence of nonnegative functions in $C_c^\infty(M)$ with  $\varphi_n \to v_-$ with respect to $\|\cdot\|_{W^1}$ (this can be proven along the same lines as \cite[Lemma~5.4]{Gri}). Recall that $\mu = \Psi {\rm vol}_g$, $\as{\cdot,\cdot}_{\rho \mu}$ denotes the inner product of $L^2(M,\rho \mu)$ and $\as{\cdot,\cdot}$ denotes the dual pairing of $\mathcal D'(M)$. Using the distributional Definition of $\cL$ and $\partial_t v \geq \cL v$, we obtain
\begin{align*}
 \as{\partial_t v, v_-}_{\rho \mu} &= \lim_{n \to \infty}\as{\partial_t v, \varphi_n}_{\rho \mu}\\
 &= \lim_{n \to \infty}\as{\partial_t v, \rho \Psi \varphi_n} \\
 &\geq  \lim_{n \to \infty} \as{\cL v, \rho\Psi \varphi_n} \\
 &=  -\lim_{n \to \infty} \Qt(v,\varphi_n)\\
 &= -\Qt(v,v_-).
\end{align*} 
Since $\Qt$ is a Dirichlet form, we have $v_+ \in D(\Qt)$ and $\Qt(v_+,v_-) \leq 0$. Therefore, the above amounts to 
$$ \as{\partial_t v, v_-}_{\rho\mu} \geq \Qt(v_-) \geq 0.$$
We shall see below that  $\partial_t \|v_-\|_2^2 = -2 \as{\partial_t v, v_-}_{\rho \mu}$, which implies $\partial_t \|v_-\|_2^2 \leq 0$. Hence,  the function $t \mapsto \|v(t)_-\|_2^2$ is decreasing in time. Since we assumed $v(t)_- \to 0$ in $L^2(M,\rho\mu)$, as $t \to 0+$, we obtain $v_- = 0$ on $(0,T)$ and we arrive at the conclusion $v \geq 0$.

It remains to prove $\partial_t \|v_-\|_2^2 = -2 \as{\partial_t v, v_-}_{\rho \mu}$. To this end, consider the $C^1$-function
$$\varphi:\R \to \R, \varphi(x) = \begin{cases}
                                     x^2 &\text{if } x \leq 0,\\
                                     0 &\text{else.}
                                    \end{cases}
$$
Its derivative satisfies
$$\varphi'(x) = \begin{cases}
                                     2x &\text{if } x \leq 0,\\
                                     0 &\text{else.}
                                    \end{cases} $$
An elementary computation shows that for any $x,y\in \R$ we have
$$\varphi(y) = \varphi(x) + \varphi'(x)(y-x) + R(x,y)$$
with a remainder $R$ that satisfies $|R(x,y)| \leq (x-y)^2$. We obtain
\begin{align*} 
\int_M ((v(s+h)_-)^2 &- (v(s)_-)^2)\rho d\mu = \int_M \left(\varphi(v(s+h)) - \varphi(v(s))\right)\rho d\mu \\
&=  \int_M \varphi'(v(s)) (v(s+h) - v(s)) \rho d\mu + \int_M R(v(s+h),v(s))\rho d\mu\\
&= -2 \int_M  v(s)_- (v(s+h) - v(s))\rho d\mu + \int_M R(v(s+h),v(s))\rho d\mu.
\end{align*}
Since $\partial_t v$ exists in $L^2(M,\mu)$ as a strong limit, we also have
\begin{align*}
 \frac{1}{h} \left| \int_M R(v(s+h),v(s)) \rho d\mu \right|&\leq \frac{1}{h} \int_M (v(s+h) - v(s))^2 \rho d\mu\\
 &= h \left\| \frac{v(s+h) - v(s)}{h} \right\|_2^2 \to 0, \text{ as } h \to 0,
\end{align*}
and
$$\frac{1}{h} \int_M  v(s)_- (v(s+h) - v(s))\rho d\mu \to  \as{(\partial_t v)(s), v(s)_-}_{\rho \mu}, \text{ as } h \to 0. $$
Combining these considerations shows  $\partial_t \|v_-\|_2^2 = -2 \as{\partial_t v, v_-}_{\rho \mu}$ and finishes the proof.
\end{proof}

\begin{remark}
 The previous lemma only relies on the fact that we work on a weighted manifold at one place. We use the identity  $-\as{\cL u,\rho \Psi\varphi} = \Qt(u,\varphi)$ for $u \in D(\Qt)$ and $\varphi \in C_c^\infty(M)$. Suppose that $\mathcal{E}$ and $\ow{\mathcal{E}}$ are Dirichlet forms (on an arbitrary $L^2$-space) such that $\ow{\mathcal{E}}$ extends $\mathcal{E}$ and $D(\ow{\mathcal{E}})$ is an order ideal of  $D(\mathcal{E})$ (in this case $\ow{\mathcal E}$ is called Silverstein extension of $\mathcal{E}$, cf. the previous remark). Our proof of Lemma~\ref{lemma:boundary inequality}  also works for the pair $\mathcal{E}$ and $\ow{\mathcal{E}}$, i.e., for $u \in D(\ow{\mathcal{E}})$ the inequality $u \geq 0 \mod D(\mathcal{E})$ is equivalent to $u_- \in D(\mathcal{E})$.  If, furthermore, one  takes the inequality
$$\as{\partial_t u,\varphi} \geq - \ow{\mathcal{E}}(u,\varphi) $$
for all nonnegative $\varphi$ in a suitable core of $D(\mathcal{E})$ as a replacement for the distributional inequality $\partial_t u \geq \cL u$, then  one can prove a parabolic maximum principle for the pair of forms $\mathcal{E}$ and $\ow{\mathcal{E}}$ along the same lines as above. In this sense our proof of the lemma is not only an extension of \cite[Theorem~5.16]{Gri} to the case when $V \neq 0$, but to general Dirichlet forms.
\end{remark}

\begin{theorem}[Elliptic maximum principle]\label{theorem:elliptic maximum principle}
Let $f \in W^1(M,\rho + V)$ with $$f \geq 0 \mod W^1_0(M,\rho + V).$$ If for some $\alpha > 0$ it satisfies $(\alpha - \cL) f \geq 0$, then  $f \geq 0$.
\end{theorem}
\begin{proof}
 Lemma~\ref{lemma:boundary inequality} implies $f_- \in W^1_0(M,\rho + V)$. As  we have seen in the proof of Lemma~\ref{theorem:parabolic maximum principle}, we can therefore choose a  sequence $(\varphi_n)$ of nonnegative functions in $C_c^\infty(M)$ that converges to $f_-$ in $W^1_0(M,\rho + V)$. Using the definition of $\cL$ we obtain
 \begin{align*}
  \Qt(f,f_-) &= \lim_{n\to \infty} \Qt(f,\varphi_n) = -\lim_{n\to \infty}  \as{\cL f,\Psi \rho \varphi_n} \\
  &\geq  -\alpha \lim_{n\to \infty}  \as{f,\Psi \rho\varphi_n} = - \alpha \as{f,f_-}_{\rho \mu} = \alpha \|f_-\|_2^2.
 \end{align*}
 Since $\Qt$ is a Dirichlet form, we also have $\Qt(f_+,f_-) \leq 0$ and therefore $\Qt(f,f_-) \leq 0$. Combined with the previous computation this shows $f_- = 0$.
 \end{proof}
\begin{remark}
 This maximum principle also holds true for a Dirichlet form and a Silverstein extension, cf. the remark after the proof of Theorem~\ref{theorem:parabolic maximum principle}.
\end{remark}

\section{Proof of the properties of $N_\alpha$ and $H_t$} \label{section:proof of the main theorems}

This section is devoted to proving Theorem~\ref{theorem:main result}. For the proof we need several auxiliary lemmas. Let $\Omega \subseteq M$ an open subset. We denote by $(T_t^\Omega)$ the $L^2(\Omega,\rho\mu)$-semigroup of the Dirichlet form 
$$\Qr^\Omega(f,g) = \int_\Omega \as{\nabla f,\nabla g} d \mu + \int_\Omega Vfg d\mu$$
with domain $D(\Qr^\Omega) = W_0^1(\Omega,\rho + V)$. We extend it to $f \in L^2(M,\rho \mu)$ by letting
$$T_t^\Omega f := \begin{cases} T^\Omega_t f|_\Omega &\text{on } \Omega,\\
                   0 &\text{on } M\setminus \Omega.
                  \end{cases}
$$
Similarly, we write $(G_\alpha^\Omega)$ for the $L^2(\Omega,\rho \mu)$-resolvent of $\Qr^\Omega$ and extend it to $L^2(M,\rho\mu)$ 

\begin{lemma}[Duhamel's principle] \label{lemma:duhamel's principle}
 Let $\Omega \subseteq M$ be an open relatively compact subset and let $f \in C^\infty(M)$. Let $u \in C^\infty((0,\infty) \times M) \cap C([0,\infty) \times M)$ nonnegative. If $\partial_t u - \cL u \geq f$ on $(0,\infty) \times \Omega$, then 
 $$u_t\geq T_t^\Omega u_0 + \int_0^t T^\Omega_s f ds.$$
\end{lemma}
\begin{proof}
  We prove that the restriction of the function
  $$t \mapsto v_t:= u_t - \left( T_t^\Omega u_0 + \int_0^t T^\Omega_s f\, ds\right)$$
  to $L^2(\Omega,\rho\mu)$ satisfies all the assumptions of 
  Theorem~\ref{theorem:parabolic maximum principle} (when applied on the weighted manifold $(\Omega,g,\mu)$). The nonnegativity of $u$ and the vanishing of $T_t^\Omega u_0$ and $T^\Omega_t f$ outside of $\Omega$ then imply the claim.
  
  Since $\Omega$ is relatively compact  and $u \in C^\infty((0,\infty) \times M)\cap C([0,\infty) \times M)$ the following holds true:
  \begin{itemize}
   \item $\partial_t u$ exists as a strong limit in $L^2(\Omega,\rho\mu)$,
   \item $u(t,\cdot) \to u_0$ in $L^2(\Omega,\rho\mu)$, as $t \to 0+$,
   \item $u|_\Omega \in W^1(\Omega, \rho + V)$.
  \end{itemize}
  Moreover, $u_0$ and $f$ are continuous on $\overline \Omega$  so that their restrictions to $\Omega$ belong to $L^2(\Omega,\rho\mu)$. The semigroup $(T_t^\Omega)$ is strongly continuous on $L^2(\Omega,\mu)$. Therefore, 
  $$T_t^\Omega u_0 + \int_0^t T^\Omega_s f ds \to u_0\mbox{  in $L^2(\Omega,\rho\mu)$, as $t \to 0+$.}$$
  This shows $\lim_{t \to 0+} v(t) = 0$ in $L^2(\Omega,\rho\mu)$.
  
  If we denote by $\Lr^\Omega$ the self-adjoint operator associated with $\Qr^\Omega$, it follows from standard semigroup theory that $T_t^\Omega u_0 \in D(\Lr^\Omega)$ and $\int_0^t T^\Omega_s f ds \in D(\Lr^\Omega)$. Since $u \geq 0$ and $D(\Lr^\Omega) \subseteq D(\Qr^\Omega) = W^1_0(\Omega,\rho+V)$, this implies $v \geq 0 \mod W^1_0(\Omega,\rho+V).$  By semigroup theory we also have $\Lr^\Omega T_t^\Omega u_0 = \partial_t (T_t^\Omega u_0)$  and 
 $$\Lr^\Omega \int_0^t T^\Omega_s f ds = T^\Omega_t f - f = \partial_t \int_0^t T^\Omega_s f ds - f,$$
 where the time derivatives exist as strong limits in $L^2(\Omega,\rho\mu)$. Therefore, $\partial_t v$ exists strongly in $L^2(\Omega,\rho\mu)$. Using that the operator $\Lr^\Omega$ is a restriction of $\cL$, these computations and the pointwise assumption $\partial_t u - \cL u \geq f$ on $\Omega$ show
 \begin{align*}
  \partial_t  v &= \partial_t  u -  \partial_t  T^\Omega_t u_0  -     \partial_t \int_0^t T^\Omega_s f ds\\
  &\geq \cL  u + f - \cL T_t^\Omega u_0 - \cL \int_0^t T^\Omega_s f ds - f\\
  &= \cL v
 \end{align*}
 in the sense of $\mathcal{D}'(\Omega)$. 
 Thus, we confirmed that $v$ satisfies all the assumptions of 
 Theorem~\ref{theorem:parabolic maximum principle} and we arrive at $v \geq 0$.
\end{proof}
The following lemma is an extension of part of \cite[Theorem~7.13]{Gri} to integrals of semigroups and the case with $V \neq 0$. 

\begin{lemma}\label{lemma:smooth choice}
 Let $e,f \in C_c^\infty(M)$ and let $u:(0,\infty) \to L^2(M,\rho \mu)$ given by
 $$u(t) = T_t^{\rho,V} e + \int_0^t T^{\rho,V}_s f ds. $$
 For each $t > 0$ we have $u(t) \in C^\infty(M)$. Moreover, the function $\tilde u:(0,\infty) \times M \to \R$, $(t,x) \mapsto u(t)(x)$ belongs to $C^\infty((0,\infty) \times M) \cap C([0,\infty) \times M)$  and satisfies
 $$\begin{cases}
    (\partial_t - \cL) \tilde u = f,\\
    \tilde u_0 = e.
   \end{cases}
$$
\end{lemma}

\begin{proof}
By standard semigroup theory the function $u$ is continuous and continuously differentiable (strongly in $L^2(M,\rho\mu)$). Let $u'$ denote its derivative. Lemma~\ref{lemma:measurable choice2} yields the existence of $\tilde u \in L^1_\mathrm{loc}((0,\infty) \times M)$ with $\partial_t \tilde u \in L_\mathrm{loc}^1((0,\infty) \times M)$ such that $\tilde u_t = u(t)$ for $\lambda$-a.e. $t > 0$  and $(\partial_t \tilde  u)_s = u'(s)$ for $\lambda$-a.e. $s>0$. Moreover, since the generator of $T^{\rho,V}_t$ is a restriction of $\cL$ we have $u' - \cL u = f$ in $L^2(M,\rho \mu)$. Together with the  properties of $\tilde u$ this implies $(\partial_t - \cL) \tilde u = f$ in the sense of distributions. We infer from Lemma~\ref{lemma:parabolic regularity} that $\tilde u \in C^\infty((0,\infty) \times M)$ (or more precisely it has a smooth version, which we consider from now on). The continuity of $u$ and the smoothness of $\tilde u$ imply $u(t) = \tilde u_t$ in $L^2(M,\rho\mu)$ for all $t > 0$.

It remains to prove the result on the initial value of $\tilde u$. Recall that $\Lr$ is the self-adjoint operator associated with $Q_{\rho,V}$ (the generator of $(T^{\rho,V}_t)$). Let $\mu_e,\mu_f$ the spectral measures of $\Lr$ at $e,f$, respectively. Let $m\in \N_0$ arbitrary. Since $e,f \in C_c^\infty(M)$,  we have $e,f \in D(\Lr^m)$  and the spectral calculus of $\Lr$ shows $T_t e, \int_0^t   T_s fds \in D(\Lr^m)$,
 $$\|\Lr^m(T_t e - e)\|_2^2 = \int_{-\infty}^0 \lambda^{2m} (e^{\lambda t} - 1)^2 d \mu_e $$
 and
 $$\|\Lr^m\int_0^t   T_s fds\|_2^2 = \int_{-\infty}^0 \lambda^{2(m-1)} (e^{\lambda t} - 1)^2 d \mu_f. $$
 For any $t> 0$ the functions 
 $$f_t:(-\infty,0] \to \R, f_t(\lambda) = \lambda^{2(m-1)} (e^{\lambda t} - 1)^2$$
 satisfy $|f_t| \leq t|\lambda|^{2m}$ and since $e,f  \in D(\Lr^m)$, the function $\lambda \mapsto |\lambda|^{2m}$ is $\mu_e$ and $\mu_f$ integrable. Hence, it follows from Lebesgue's dominated convergence theorem that
 $$\lim_{t \to 0+} \|\Lr^m(T_t e - e)\|_2^2 = \lim_{t \to 0+}\|\Lr^m\int_0^t   T_s fds\|_2^2 = 0.$$
  The operator $\Lr$ is a restriction of $\cL$   so that these computations show $\|\cL^m(\tilde u_t - e)\|_2 \to 0$, as $t \to 0+$, for every $m \in \N_0$.  The local Sobolev embedding theorem, see Lemma~\ref{lemma: local sobolev embeddings}, implies that this convergence also holds locally uniformly. In particular, we obtain $\tilde u_0 = e$ and we arrive at the desired claim. 
\end{proof}

 In the following lemma we denote by $A_t$ the positivity preserving operator  
 $$A_t:L^2(M,\rho\mu) \to L^2(M,\rho \mu), \quad A_t f = \int_0^t T_s^{\rho, V} f ds,$$
 and its extension to $L^+(M)$.
\begin{lemma}\label{lemma:weakstarcontinuity}
 Let $f \in L^+(M)$ such that  $A_t f \in L^\infty(M)$ for all $t> 0$.  Then the map $(0,\infty) \to L^\infty(M)$, $t \mapsto A_t f$ is weak-$*$-continuous.
\end{lemma}
\begin{proof}
 We denote by $(\cdot,\cdot)$ the dual pairing between $L^\infty(M)$ and $L^1(M,\rho \mu)$.  For given  $g \in L^1(M,\rho \mu)$ and $t_0,\varepsilon > 0$ we choose $ \tilde f \in L^2(M, \rho \mu) \cap L^\infty(M)$ with $0 \leq \tilde f \leq f$ such that 
 $$0 \leq (A_{t_0 + 1} f - A_{t_0 + 1} \tilde f, |g|) < \varepsilon. $$
 Such an $\tilde f$ exists by the definition of the extension of $A_t$ to $L^+(M)$ and the assumption $A_t f \in L^\infty(M)$ for all $t > 0$. 
 Since the semigroup $(\Tr)$ is positivity preserving, for $h \in L^+(M)$ the map $(0,\infty) \to L^+(M), t \mapsto A_t h$ is monotone increasing in the sense of the order on $L^+(M)$. Using that $A_t$ is linear on $\mathrm{dom}(A_t)$, this monotonicity implies for $0 < t < t_0+1$ that
 $$|(A_t f - A_t \tilde{f}, |g|)| = |(A_t(f - \tilde f),|g|)| \leq |(A_{t_0 + 1}(f - \tilde f),|g|)| = (A_{t_0 + 1} f - A_{t_0 + 1} \tilde f, |g|) < \varepsilon.  $$
 For $0 < t < t_0+1$ we obtain
 $$|(A_t f - A_{t_0} f,g)| \leq (|A_t \tilde f - A_{t_0} \tilde f|, |g|) +  2 \varepsilon  = \left|\int_{t_0}^t ( T^{\rho,V}_s \tilde f, |g|) ds\right| + 2\varepsilon.   $$
 Since $\tilde f$ was chosen in $L^\infty(M)$, the map $(0,\infty) \to L^\infty(M), s \mapsto  T^{\rho,V}_s \tilde f$ is weak-$*$-continuous. Hence, it follows from the above computation that for $t$ sufficiently close to $t_0$ we have $|(A_t f - A_{t_0} f,g)| < 3\varepsilon$. Since $\varepsilon$ was arbitrary, this finishes the proof. 
\end{proof}

With all these preparations at hand we can now proof Theorem~\ref{theorem:main result}. Recall that  $H_t = \Tr 1 + \int_0^t T^{\rho,V}_s \hat V ds$ and $N_\alpha = \alpha \Gr 1 + \Gr \hat V$ with $\hat V = V / \rho$.

%


 \begin{proof}[Proof of Theorem~\ref{theorem:main result}]
 In this proof we write $G_\alpha$ for $\Gr$ and $T_t$ for $\Tr$ to shorten notation.
 
(a):  The bound $0\leq H$ is obvious as the involved operators are positivity preserving. Let $e,f \in   C_c^\infty(M)$ with $0 \leq e \leq 1$ and $0 \leq f \leq \hat V$ and let $\Omega \subseteq M$ open and relatively compact. A function $u$ as in the maximality statement  satisfies $(\partial_t - \cL)(1-u) \geq \hat V$ and $(1-u)_0 = 1$. Thus, Duhamel's principle, Lemma~\ref{lemma:duhamel's principle},  shows
 $$1 - u_t \geq T_t^\Omega 1 + \int_0^t T_s^\Omega \hat V ds \geq T_t^\Omega e + \int_0^t T_s^\Omega f ds. $$
 Letting $\Omega \nearrow M$ and using Lemma~\ref{lemma:convergence of semigroups} this implies
 \begin{align}\label{inequality:fundamental}
  1 - u_t \geq T_t e + \int_0^t T_s f ds. \tag{$\heartsuit$}
 \end{align}
 Moreover, letting $e \nearrow 1$ and $f \nearrow \hat V$ we arrive at $1 - u_t \geq H_t$ for all $t > 0$. Since the extension of the semigroup $T_t^{\rho,V}$ to $L^\infty(M)$ is weak-$*$-continuous, this observation and Lemma~\ref{lemma:weakstarcontinuity} yield that $(0,\infty) \to L^\infty(M)$, $t \mapsto H_t$ is weak-$*$-continuous.
 
Lemma~\ref{lemma:smooth choice} implies that for each $t > 0$ we have
 $$H^{e,f}_t := T_t e + \int_0^t   T_s fds \in C^\infty(M)$$
 and that the function $H^{e,f}$  given by $(t,x) \mapsto H_t^{e,f}(x)$ satisfies $H^{e,f} \in C^\infty((0,\infty) \times M)$.

Let now $(e_n),(f_n)$  increasing sequences in $C_c^\infty(M)$ with $e_n \to 1$ pointwise and $f_n \to \hat V$ pointwise. Since $(\Tr)$ is positivity preserving and Inequality~\ref{inequality:fundamental} holds for all $e_n,f_n$, for each $t> 0$ we have
$$H^{e_n,f_n}_t \leq H^{e_{n+1},f_{n+1}}_t \leq 1-u_t\,\mu\text{-a.e.}$$
The smoothness of $H^{e_n,f_n}$ and $1-u$ implies that these inequalities hold for all $(t,x) \in (0,\infty)\times M$. In particular, the limit $h:= \lim_{n\to \infty}H^{e_n,f_n}$ exists pointwise.

It follows from the definition of $H$ that for each $t> 0$ we have $h_t = H_t$ $\mu$-a.e. In particular, $h$ is a measurable version of $H$ on $(0,\infty) \times M$  and the bound $H^{e_n,f_n} \leq 1 - u$ yields $h \leq 1-u$.  
 
 From the monotone convergence theorem and the bound $0 \leq h \leq 1$ we further obtain $H^{e_n,f_n} \to h$ in $L^1_{\rm loc}((0,\infty) \times M)$. Since $L^1_{\rm loc}((0,\infty) \times M)$ continuously embeds into $\mathcal{D}'((0,\infty) \times M)$ and since the distributional version of $\partial_t - \cL$ is continuous with respect to $\mathcal{D}'$-convergence, this and Lemma~\ref{lemma:smooth choice} imply 
 $$(\partial_t - \cL)h = \lim_{n \to \infty}(\partial_t - \cL)H^{e_n,f_n} = \lim_{n \to \infty}f_n = \hat V, $$
 in $\mathcal{D}'((0,\infty) \times M)$. It follows from Lemma~\ref{lemma:parabolic regularity} that $h$ has a version   $ \tilde h \in C^\infty((0,\infty) \times M)$. From Fubini's theorem we infer that for a.e. $t> 0$ we have $h_t = \tilde h_t$ $\mu$-a.e. This implies $\tilde h_t = H_t$ in $L^+(M)$ for $\lambda$-a.e. $t> 0$.  Since $\tilde h$ is smooth and $t \mapsto H_t$ is weak-$*$-continuous on $L^\infty(M)$, we arrive at $\tilde h_t = H_t$ in $L^+(M)$ for all $t >0$, the claimed smoothness of $H_t$. The inequality $u_t \leq 1 - H_t$ has been proven along the way.
 
 It remains to show the result on the initial value of $H$. For this purpose it suffices to prove $\tilde{h} \in C([0,\infty) \times M)$ and $\tilde h_0 = 1$. Let $K \subseteq M$ compact and suppose $e = 1$ on $K$. Since for every $t> 0$ we have $H_t^{e,f} \leq H_t \leq 1$ pointwise a.e. and $(t,x) \to H^{e,f}_t(x)$ is smooth,  for every $t > 0$ we obtain the pointwise estimate
 $$|1 - \tilde{h}_t| \leq |1 - H^{e,f}_t|.  $$
 According to Lemma~\ref{lemma:smooth choice} we have $H^{e,f}_t \to e$ locally uniformly, as $t \to 0+$. This implies $\tilde{h}_t \to 0$ locally uniformly on $K$, as $t \to 0+$. Since $K$ was arbitrary, this proves $\tilde{h} \in C([0,\infty) \times M)$ and $\tilde h_0 = 1$.

(b): 
%
%
%
Let $g \in  C^\infty(M)$ with $g \leq 1$  and $(\alpha-\cL) g \leq 0$ be given. For $\Omega \subseteq M$ relatively compact and $e,f \in C_c^\infty(M)$ with $0 \leq e\leq 1$ and $0 \leq f \leq \hat V$ we let
$$h^\Omega_{e,f} = 1  - \alpha G^\Omega_\alpha e -  G^\Omega_\alpha f - g|_\Omega. $$
 Then $h^\Omega_{e,f} \in W^1(\Omega,\rho + V)$. Since $\alpha G^\Omega_\alpha e +  G^\Omega_\alpha f \in W^1_0(\Omega, + V)$ and $g \leq 1$, we have $h^\Omega_{e,f} \geq 0 \mod W^1_0(\Omega,\rho + V)$. Moreover, the inequality $(\alpha - \cL) g \leq 0$ and that $\cL$ is an extension of the generator of $G_\alpha^\Omega$ yield 
$$(\alpha - \cL) h^\Omega_{e,f} = \alpha  + \hat V  - \alpha e - f  - (\alpha - \cL) g \geq 0 \text{ on } \Omega.$$
  With this at hand, Lemma~\ref{theorem:elliptic maximum principle} shows $h^\Omega_{e,f} \geq 0$. Letting first $\Omega \nearrow M$ and using Lemma~\ref{lemma:convergence of semigroups} and letting $e \nearrow 1$, $f \nearrow \hat V$ afterwards shows $1-N_\alpha - g \geq 0$. Let now $g = 0$.  Since the  convergence of $h^\Omega_{e,f}$ is monotone decreasing and $h^\Omega_{e,f}$ is nonnegative, we also obtain $h^\Omega_{e,f} \to 1 - N_\alpha $ in $\mathcal{D}'(M)$. This implies
  $$(\alpha - \cL) (1 - N_\alpha)= \lim_{f \nearrow V}\lim_{e \nearrow 1}\lim_{\Omega \nearrow M}(\alpha - \cL)h^\Omega_{e,f} =  0.$$
   Lemma~\ref{lemma: local sobolev embeddings} yields that $1 - N_\alpha$ and hence also $N_\alpha$ is smooth. This finishes the proof of (b). 

(c): For  $H$ and $N_\alpha$ we choose the smooth versions as in (a) and (b). The $L^2$-resolvent $(G_\alpha)$ is the Laplace transform of $(T_t)$ (in the Bochner sense), i.e., for all $\alpha > 0$ and $f \in L^2(M,\rho\mu)$ we have
$$\int_0^\infty e^{-t\alpha} T_t f dt = G_\alpha f. $$
Let now $e,f \in C_c(X)$ with $0 \leq e \leq 1$ and $0 \leq f \leq V$, respectively, and let $H^{e,f}_t \in C^\infty(M)$ be as defined as  in the proof of (a).   Fubini's theorem   yields 
\begin{align*}
 \int_0^\infty e^{-t\alpha} H^{e,f}_t  dt =  \int_0^\infty e^{-t\alpha} T_t e dt + \int_0^\infty \int_0^t e^{-t\alpha} T_s f dsdt =  G_\alpha e +  \frac{1}{\alpha} G_\alpha f, 
\end{align*}
where the identity holds in $L^2(M,\rho \mu)$. Since both side of the equation are smooth, it also holds pointwise everywhere on $M$. Letting $e \nearrow 1$ and $f \nearrow V$ we have that  $H^{e,f}  \nearrow H$ $\lambda \otimes \mu$-a.e. (cf. proof of (a))  and $G_\alpha e +  \frac{1}{\alpha} G_\alpha f \nearrow \alpha^{-1}N_\alpha$ $\mu$-a.e. With this at hand, for nonnegative $\varphi \in \mathcal{D}(M)$ we conclude with the help of the monotone convergence theorem and Fubini's theorem that
$$\langle \int_0^\infty e^{-t\alpha} H_t  dt,\varphi \rangle = \frac{1}{\alpha} \langle N_\alpha, \varphi \rangle.  $$
Since $\int_0^\infty e^{-t\alpha} H_t dt$ and $N_\alpha$ are continuous, we arrive at the desired identity.

(d): We only prove the statement for $N_\alpha$, the statement for $H$ then follows with the help of assertion (c). For $\alpha > 0$ we consider the smooth function $u = 1 - N_\alpha$. We prove that $u(x) = 0$ for some $x \in M$ implies $u = 0$ on $M$. Since $M$ is connected and the set of zeros $N = \{x \in M \mid u(x) = 0\}$ is closed, it suffices to show that $N$ is also open. 

According to (a) we have $(\alpha - \cL) u = 0$, or equivalently, 
$$-\Delta u + (\alpha \rho + V)u = 0.$$
Let $x \in N$  and let $\varphi:U \to  V$ be  a chart with $x \in U \subseteq M$ and $V \subseteq \R^d$. The equation for $u$ implies that in local coordinates the smooth function $\bar u = u \circ \varphi^{-1}:V \to \R$ satisfies 
$$D\bar u := -\sum_{i,j} a_{ij} \partial_i \partial_j \bar u + \sum_i b_i \partial_i \bar u + c \bar u = 0,$$
where $a_{ij},b_i$ and $c$ are smooth functions on $V$ and $\partial_i$ are the ordinary partial derivatives. Moreover, the $0$-th order coefficient $c$ is nonnegative and by shrinking $U$ if necessary  the coefficient matrix $(a_{ij})$ can be chosen to be uniformly elliptic. For such elliptic operators it is well-known that smooth  functions $v:V \to [0,\infty)$ with $Dv = 0$ and $v(x) = 0$ for some $x \in V$ vanish identically on $V$, see e.g. \cite[Theorem 6.4.4]{Eva}. We conclude $\bar u = 0$ on $V$ so that $u = 0$ on $U$. This shows that the set of zeros of $u$ is open and finishes the proof.
\end{proof}

Our proofs would also allow for real-valued potentials without a definite sign.  However, in this case it seems to be technically hard to even define  $H$ and $N_\alpha$. We shortly discuss this in the following remark.
 
\begin{remark}
Our proof of Theorem~\ref{theorem:main result} shows that in principle we could drop the assumption that the smooth potential $V$ is nonnegative and allow sufficiently small negative parts.  More precisely, if the negative part $V_-$ belongs to the extended Kato class of the operator $L_{\rho,V_+}$ and satisfies $c_\alpha(V_-) < 1$ for some $\alpha > 0$ in the sense of \cite{SV}, the form $\Qr$ is closed on the domain $W^1_0(M,\rho + V_+)$ and the associated semigroup and resolvent are positivity preserving and map $L^\infty(M)$ to $L^\infty(M)$. In this case, also the main technical lemmas, namely the maximum principles Lemma~\ref{theorem:parabolic maximum principle} and Lemma~\ref{theorem:elliptic maximum principle}, hold true for paths in $W^1(M,\rho + V_+)$. In their proofs we only used that the form $Q_{\rho,V}$ satisfies the first Beurling-Deny criterion and not the second. However, it is harder to control the quantity  $\int_0^t T_s V ds$. If we let $H^\Omega_t = T_t^\Omega 1 + \int_0^t T_t^\Omega Vds$,   Duhamel's principle (applied to $1$ and $1-H^\Omega$)  guarantees that for any compact  $\Omega \subseteq M$ and $t> 0$ we have 
$$1 \geq T_t^\Omega 1 + \int_0^t T_t^\Omega Vds \geq 0.$$
In contrast to situation for nonnegative potentials, this inequality and  $T_t1 \in L^\infty(M)$ do not guarantee the finiteness of 
$$\int_0^t T_t |V| ds.$$
Hence, $H$ (and also $N_\alpha$) may not be well-defined. However, if $V_-$ is small as above  and additionally satisfies $V_- \in L^q(M,\rho \mu)$ for some $1\leq q \leq \infty$, it follows from the considerations in \cite{SV} that  $\int_0^t T_t^\Omega V_-ds \in L^q(M,\rho \mu)$. In this situation, the statements of Theorem~\ref{theorem:main result} remain true. Once the existence of $H$ is settled, the above inequalities show $0 \leq H \leq 1$ after letting $\Omega \nearrow M$ and using Lemma~\ref{lemma:convergence of semigroups}. In particular, $H$ is bounded and the assertions of Theorem~\ref{theorem:non conservative}   and Theorem~\ref{theorem:conservative} can be proven with the help of Theorem~\ref{theorem:main result}.  We leave the details to the reader. Since the extended Kato condition of \cite{SV} is technical and not easily verified on manifolds, see e.g. the recent discussion in \cite{Gun}, we chose to state our main results for nonnegative potentials only. 
 \end{remark}

\section{Applications}

In this section we discuss several applications of the generalized conservation property. We show that the generalized conservation property implies Markov uniqueness. On model manifolds we provide a characterization of the generalized conservation property in terms of volume growth and we show that on complete manifolds there always exists a potential making the forms conservative in the generalized sense.

\subsection{Markov uniqueness}

 It is  well known that on a stochastically complete manifold we have $W^1(M) = W^1_0(M)$. In   terms of associated stochastic processes this means that there is only one Brownian motion on $M$. It is also known from Dirichlet form theory that the conservativeness of $Q_{\rho,0}$ implies $W^1(M,\rho) = W^1_0(M,\rho)$ and that in the finite measure case (i.e. when $\rho \in L^1(M,\mu)$) both properties (uniqueness of Brownian motion and $W^1(M,\rho) = W^1_0(M,\rho)$) are equivalent, see e.g. \cite{Kuw,HKLMS}. The following theorem provides an extension to the case when $V$ does not vanish.

 \begin{theorem} \label{theorem:markov uniqueness} Of the following assertions (i) always implies (ii). If $\rho+ V\in L^1(M,\mu)$, then they are equivalent.
 \begin{itemize}
  \item[(i)] $\Qr$ is conservative in the generalized sense.
  \item[(ii)] $W^1(M,\rho + V) = W^1_0(M,\rho + V)$.
 \end{itemize}
 \end{theorem}
\begin{proof}
(i) $\Rightarrow$ (ii):  We assume that $W^1(M,\rho + V) \neq  W^1_0(M,\rho + V)$ and denote by $\Qt$ the Dirichlet form with domain $W^1(M,\rho  + V)$ that was introduced in Section~\ref{section:maximum principles}. Since essentially bounded functions are dense in the domains of Dirichlet forms, see e.g. \cite[Theorem~1.4.2]{FOT}, there exists an essentially bounded $h \in  W^1(M,\rho + V) \setminus W^1_0(M,\rho + V)$ with $\|h\|_\infty \leq 1$.  Let $h_0$ be the Hilbert space projection of $h$ in $(W^1(M,\rho + V),\|\cdot\|_{W^1})$ onto the closed subspace $W^1_0(M,\rho + V)$. We prove that $h_r := h - h_0$ is essentially bounded and satisfies $(1-\cL)h_r = 0$. It then follows from Lemma~\ref{lemma: local sobolev embeddings} that $h_r$ is smooth. By Theorem~\ref{theorem:non conservative} the form $\Qr$ is not conservative in the generalized sense, a contradiction.

Since $h_0$ is the orthogonal projection onto $W^1_0(M,\rho+V)$, $h_r$ is orthogonal to all $\varphi \in \mathcal{D}(M)$ with respect to the inner product induced by $\|\cdot\|_{W^1}$ on $W^1(M,\rho  + V)$. Hence, for all $\varphi \in \mathcal{D}(M)$ we obtain
$$\int_M \as{\nabla h_r, \nabla \varphi } d\mu + \int_M h_r \varphi V d\mu +\int_M h_r \varphi \rho d\mu = 0. $$
By definition of $\cL$ this is nothing more than saying     $(1-\cL)h_r = 0$ in the sense of distributions.

It remains to prove the boundedness of $h_r$. By the projection theorem in Hilbert spaces it satisfies
$$\|h_r\|_{W^1} = \|h - h_0\|_{W^1} = \inf \{ \|h - \psi\|_{W^1} \mid \psi \in W^1_0(M,\rho + V)\}.$$
Let $\psi_n$ a sequence in $W^1_0(M,\rho + V)$ converging to $h_0$. We prove that 
$$\varphi_n:= h -((h - \psi_n) \wedge 1) \vee(-1)$$
 belongs to $W^1_0(M,\rho + V)$ and converges to $h_0$. From this the claimed boundedness of $h_r = h - h_0$ follows.     
 
 Since $\|\cdot\|^2_{W^1}$  with domain $W^1(M,\rho + V)$ is a Dirichlet form, the functions $\varphi_n$ belong to $W^1(M,\rho + V)$. Moreover, using $\|h\|_\infty \leq 1$ shows $|\varphi_n| \leq |\psi_n|$. Since $\psi_n \in W^1_0(M,\rho + V)$ and  $W^1_0(M,\rho + V)$ is an order ideal in $W^1(M,\rho + V)$, see Lemma~\ref{lemma:order ideal}, we obtain $\varphi_n \in W^1_0(M,\rho + V)$. Using again that $\|\cdot\|^2_{W^1}$  is a Dirichlet form we arrive at
 $$\|h - \psi_n\|^2_{W^1}\geq \|((h - \psi_n) \wedge 1) \vee(-1)\|^2_{W^1} =\|h -  \varphi_n\|^2_{W^1} = \|h_0 - \varphi_n\|^2 + \|h - h_0\|^2_{W^1}.$$
 This shows $\varphi_n \to h_0$ and finishes the proof of (i) $\Rightarrow$ (ii).
 
 Now assume $\rho + V \in L^1(M,\mu)$.  (ii) $\Rightarrow$ (i):  Due to    $\rho + V \in L^1(M,\mu)$  we have $1 \in   W^1(M,\rho + V) = W^1_0(M,\rho + V)$.  This implies that the Dirichlet form $Q_{\rho + V,0}$ on $L^2(M,\rho + V)$ is recurrent, see \cite[Theorem~1.6.3]{FOT}.  Moreover, it is well known that recurrence implies conservativeness, see e.g. \cite[Theorem~1.6.5 and Theorem~1.6.6]{FOT}, i.e., $Q_{\rho + V,0}$ is conservative. According to Theorem~\ref{theorem:criteria for generalized conservation}    this shows that $Q_{\rho,V}$ is conservative in the generalized sense.
\end{proof}

 \begin{remark}
 \begin{enumerate}
   \item  An abstract version of this theorem for general Dirichlet forms is \cite[Corollary 4.58]{Schmi}. The above proof is basically the one in \cite{Schmi} adapted to our situation.
   
  \item It follows with arguments involving domination of quadratic forms that the equality  $W^1(M,\rho) = W^1_0(M,\rho)$ always implies $W^1(M,\rho + V) = W^1_0(M,\rho + V)$ for nonnegative $V$, see \cite{LSW}. In particular, if $Q_{\rho,0}$ is conservative we have $W^1(M,\rho + V) = W^1_0(M,\rho + V)$. However, we would like to point out that our theorem still strengthens the known criteria for $W^1(M,\rho + V) = W^1_0(M,\rho + V)$  since there are manifolds and weights for which $T_t^{\rho,0}$ is not conservative but $T_t^{\rho,V}$ is conservative in the generalized sense, see Corollary \ref{corollary:model manifold} below.
 
 \end{enumerate}

 \end{remark}

 \subsection{Model manifolds} \label{subsection:model manifolds} In this subsection we characterize the generalized conservation criterion for model manifolds.
 
 Let $E =  (0,\infty) \times \mathbb{S}^{n-1} \cup \{0\}$ be equipped with the topology making the coordinate map $\R^n \to E, x \mapsto (|x|, |x|^{-1}x)$ if $x \neq 0$ and $0 \mapsto 0$ a homeomorphism.   A Riemannian manifold $(M,g)$ is called   $n$-dimensional model manifold (of infinite radius) if $M = E$ and  for each $(r,\theta) \in (0,\infty) \times \mathbb{S}^{n-1}$ on the tangent space $T_{(r,\theta)}M = T_r(0,\infty) \oplus T_\theta \mathbb{S}^{n-1}$ the metric takes the form 
  $$g = (dr)^2 + \sigma(r)^2 (d\theta)^2,$$
 where $\sigma:(0,\infty) \to (0,\infty)$  is smooth and $(dr)^2, (d\theta)^2$ denote the standard metrics on $T_r(0,\infty)$ respectively $T_\theta \mathbb{S}^{n-1}$. In this case, the function $\sigma$ is called the {\em scaling function} of the model.

Let $(M,g)$ be a model manifold and let $d$ be the induced geodesic distance on $M$. We call a function $f:M \to \R$   radially symmetric if $f(x) = f(y)$ for all $x,y \in M$ with $d(x,0) = d(y,0)$. In this case there exists a function $\tilde f:[0,\infty) \to \R$ such that $f(x) = \tilde f(r)$ whenever $d(x,0) = r$. In what follows we abuse notation and simply write $f(r)$ instead of $\tilde{f}(r)$. 

For simplicity  we also assume $\mu = \vo$, i.e.,   $\Psi  = 1$. If $f$ is smooth and radially symmetric, so is $\Delta  f = \mathrm{div} (\nabla f)$ and for $r > 0$ it takes the form
$$\Delta f(r) = f''(r) + (n-1)\frac{\sigma'(r)}{\sigma(r)} f'(r),$$
see e.g. \cite{Gri}. If, moreover,  $\rho$ and $V$ are radially symmetric, so is $\cL f$ and for $r>0$ it is given by
$$\cL f(r) = \rho(r)^{-1} f''(r) + \rho(r)^{-1} (n-1)\frac{\sigma'(r)}{\sigma(r)} f'(r)- \rho(r)^{-1} V(r) f(r).$$

From now on we assume that $\rho$ is radially symmetric. For $r \geq 0$ we denote by $v_{\rho}(r) =  \int_{B_r(0)} \rho d \vo$ the  volume of the $d$-ball of radius $r$ around $0$ with respect to the measure $\rho \vo$. It can be computed as
$$v(r) = \int_0^r  \int_{\mathbb{S}^{n-1}} \rho(\xi) \sigma^{n-1}(\xi)  d\xi d\theta = \omega_{n} \int_0^r \rho(\xi) \sigma^{n-1}(\xi) d\xi,$$
where $\omega_{n}$ is the surface area of $\mathbb{S}^{n-1}$ (the volume of the unit sphere in $\R^n$), see e.g. \cite{Gri}. The quantity $s(r) = \omega_n  \sigma^{n-1}(r)$ is the surface area of the sphere of radius $r$ in the  manifold $(M,g)$. We obtain the following characterization of the generalized conservation property on model manifolds.

\begin{theorem}\label{theorem:volume growth model manifold}
 Let $(M,g)$ be a model manifold with radially symmetric $\rho$ and $V$ and let $\mu = \vo$. The following assertions are equivalent.
 \begin{enumerate}[(i)]
  \item  The Dirichlet form $\Qr$ on $L^2(M,\rho \vo)$ is conservative in the generalized sense.
  \item There exists $a > 0 $ such that
 $$\int_a^\infty \frac{v_{\rho + V}(r)}{s(r)} dr = \int_a^\infty \frac{v_\rho(r)}{s(r)} dr + \int_a^\infty \frac{\int_0^r s(\xi)V(\xi) d\xi}{s(r)}dr = \infty.$$
 \end{enumerate}
\end{theorem}

\begin{remark}
\begin{enumerate}
 \item Weakly spherically symmetric graphs are discrete analogues of model manifolds. A version of the theorem for weakly spherically symmetric graphs is contained in \cite{KLW}. 
 \item Since $\Qr$ on $L^2(M,\rho \vo)$ is conservative in the generalized sense if and only if $Q_{\rho+V,0}$ on $L^2(M,(\rho + V) \vo)$ is conservative, the theorem is an extension of the well known characterization of stochastic completeness of model manifolds (the case $\rho = 1, V = 0$) in terms of their volume growth, see e.g. \cite[Proposition 3.2]{Gri1}. More precisely, it treats the case when the measure used in the integrals in the definition of the form $Q_{\rho +V,0}$ (in this case $\vo$) and  the measure of the underlying $L^2$-space (in this case $(\rho + V) \vo$) are different.
\end{enumerate}
\end{remark}

Before proving the theorem we  note the following elementary lemma on radially symmetric smooth solutions to the eigenvalue equation. We include a proof for the convenience of the reader.

\begin{lemma}\label{lemma:model manifolds}
 Let $(M,g)$ be a model manifold with radially symmetric $\rho$ and $V$ and let $\mu = \vo$. Moreover, let $f:M \to [0,\infty)$ be a smooth  radially symmetric solution to the equation $(1 - \cL) f = 0$. For all $r > r' \geq 0$ it satisfies the identity
 $$f(r) = f(r') + \int_{r'}^r \frac{1}{s(\eta)}  \int_0^\eta s(\xi)(\rho(\xi) + V(\xi)) f(\xi) d\xi d\eta. $$
 In particular, the function  $[0,\infty) \to  [0,\infty), r \mapsto f(r)$ is monotone increasing.
\end{lemma}
\begin{proof}
 We use the discussed formulas for $\cL f$ for radially symmetric functions and the identity $(n-1) \sigma'/\sigma = s'/s$ to obtain that $f$ satisfies
 $$f'' + \frac{s'}{s} f' = (\rho + V) f.$$
 Multiplying this by $s$, integrating it from $0$ to  $\eta > 0$ and using $s(0+) = \omega_n \sigma^{n-1}(0+) = 0$ yields
 $$(sf')(\eta) = \int_0^\eta s(\xi)(\rho(\xi) + V(\xi)) f(\xi) d\xi.$$
 Dividing by $s$ and integrating from $r'$ to $r$ yields
 $$f(r) = f(r') + \int_{r'}^r \frac{1}{s(\eta)}  \int_0^\eta s(\xi)(\rho(\xi) + V(\xi)) f(\xi) d\xi d\eta.  $$
 The ``In particular''-part follows from the fact that all integrands are nonnegative.
\end{proof}

\begin{proof}[Proof of Theorem~\ref{theorem:volume growth model manifold}]

 (i) $\Rightarrow$ (ii): We assume 
 $$\int_1^\infty \frac{v_{\rho + V}(r)}{s(r)} dr < \infty$$
 and use \cite[Theorem~7.5]{HKLMS} to show that the Dirichlet form $Q_{\rho + V,0}$ on $L^2(M,(\rho + V)\vo)$ is not conservative. According to Theorem~\ref{theorem:criteria for generalized conservation} this implies that $\Qr$ is not conservative in the generalized sense. Strictly speaking  \cite[Theorem~7.5]{HKLMS} can only be applied to the form   $Q_{1,0}$, i.e., the case when $\rho + V = 1$. However, the  theory developed in \cite{HKLMS} treats arbitrary Dirichlet forms  and \cite[Theorem~7.5]{HKLMS}is an application of this more general frameworkt to manifolds. It   also works when $\rho + V \neq 1$.
 
 According to \cite[Theorem~7.5]{HKLMS}  it suffices to prove the existence of a function $f \in L^1(M,(\rho + V) \vo) \cap L^2(M,(\rho + V) \vo)$ with $|\nabla f| \in L^2(M,\vo)$ and $\mathcal{L}_{\rho + V,0} f \in L^1(M,(\rho + V) \vo) \cap L^2(M,(\rho + V) \vo)$ such that
 $$\int_{M} \mathcal{L}_{\rho + V,0} f (V + \rho) d \vo = \int_{M} \Delta f d \vo \neq 0.$$
 We consider the radially symmetric  function $g$ on $M \setminus B_{1/2}(0)$ given by
 $$g(r) = \int_r^\infty \frac{1}{s(\xi)} d\xi, \quad r > \frac{1}{2}.$$
  Fubini's theorem and our volume growth condition imply   $g \in L^1(M\setminus B_{1/2}(0),(\rho + V) \vo)$.  Since $g$ is monotone decreasing with $g(r) \to 0,$ as $r \to \infty$, it also belongs to $L^2(M \setminus B_{1/2}(0),(\rho + V) \vo)$.  In particular, $g$ is finite and smooth and satisfies
  $$\Delta g(r) = g''(r) - \frac{s'(r)}{s(r)} g'(r) = 0, \quad r >\frac{1}{2}.$$
  Moreover, since $M$ is a model manifold, we have $|\nabla g| (r) = |g'(r)|$ so that
  $$\int_{M\setminus B_{1}(0)} |\nabla g|^2 d \vo =   \int_1^\infty s(r) \frac{1}{s(r)^2} dr =g(1)  < \infty.$$

   Let now $f$ be a smooth radially symmetric function on $M$ such that $f = g$ on $M\setminus B_{1/2}(0)$. The desired integrability properties of $f$ directly follow from properties of $g$ and so it remains to prove the statement on $\Delta f$. 
  Using the identities $\Delta f = s^{-1} (s f')'$ and $\Delta f(r) = 0$ for $r > 1$ we compute 
  $$ \int_M \Delta f d\vo = \int_0^\infty s(r) \Delta f (r) dr = \int_0^1 s(r) \Delta f(r) dr = f'(1) s(1) - f'(0+)s(0+) = -1. $$
  This proves (ii).
 
 (ii) $\Rightarrow$ (i):  We consider the function $f = 1 - N_1 = 1 - G^{\rho,V}_1 1 - G^{\rho,V}_1 (V / \rho)$. According to Theorem~\ref{theorem:main result} it is smooth, satisfies $0 \leq f \leq 1$ and $(1 - \cL)f = 0$.  Moreover, since $V$, $\rho$  and the constant function $1$ are  radially symmetric, it follows from routine arguments that $f$ is radially symmetric (use that the form $\Qr$ is invariant with respect to rotations in the $\mathbb{S}^{n-1}$-variable and that this invariance yields an invariance for the resolvent; for the operator theoretic details see e.g. \cite[Appendix~A]{KLW}).  Assume now that (i) does not hold. By Theorem~\ref{theorem:main result} this implies $f(x) > 0$ for all $x \in M$ and hence $f(r) > 0$ for all $r \geq 0$. According to Lemma~\ref{lemma:model manifolds} $[0,\infty) \to (0,\infty), r \mapsto f(r)$ is monotone increasing. Hence, the integral formula of  Lemma~\ref{lemma:model manifolds} implies
 $$f(r) \geq f(0) + f(0)  \int_{0}^r \frac{1}{s(\eta)}  \int_0^\eta s(\xi)(\rho(\xi) + V(\xi))  d\xi d\eta = f(0) + f(0) \int_0^r \frac{v_{\rho + V}(\eta)}{s(\eta)} d\eta. $$
 With (ii) this shows $f(r) \to \infty$, as $r \to \infty$, a contradiction to the boundedness of $f$.
\end{proof}

\begin{corollary}\label{corollary:model manifold}
Let $(M,g)$ be a model manifold with radially symmetric $\rho$ and let $\mu = \vo$. There exists a radially symmetric $V$ such that $\Qr$ is conservative in the generalized sense.
\end{corollary}
\begin{proof}
 Choose a smooth function $f:[0,\infty) \to [0,\infty)$ such that $f = 0$ on a neighborhood of $0$ and 
 $$\int_1^\infty \frac{\int_0^r s(\xi)f(\xi) d\xi}{s(r)}dr = \infty.$$
 Since the metric $d$ is smooth off the diagonal, the function $V:M \to [0,\infty)$ defined by $V(x) = f(d(0,x))$ is smooth and radially symmetric. With this choice of $V$ it follows from Theorem~\ref{theorem:volume growth model manifold} that $\Qr$ is conservative in the generalized sense.
\end{proof}

\subsection{Large potentials on complete manifolds}\label{subsection:large potentials}
 
 In this subsection we show that on a  complete weighted manifold $(M,g,\mu)$ we can choose a smooth potential $V \geq 0$ such that $Q_{1,V}$ is conservative in the genralized sense. Our proof is based on a volume growth criterion for stochastic completeness for strongly local regular Dirichlet forms from \cite{Stu1}. We refer to this reference for the terminology used in this subsection.
 
 We write $W^1_{\rm loc}(M)$ for the local first order Sobolev space, i.e., $f \in W^1_{\rm loc}(M)$ iff for every relatively compact open $\Omega \subseteq M$ there exists $\varphi \in W^1_0(M,\rho)$ (or equivalently $\varphi \in W^1(M,\rho)$)  with $f = \varphi$ on $\Omega$. Since $\rho$ is assumed to be smooth (and hence bounded and strictly positive on a neighborhood of every relatively compact open set), this definition is independent of $\rho$.

 The quadratic form $Q_{\rho,0}$ is a strongly local regular Dirichlet form on $L^2(M,\rho \mu)$ and its {\em local domain} is given by $W^1_{\rm loc}(M)$. If we denote by $\Gamma(f) = |\nabla f|^2 \mu$ the {\em energy measure} of a function $f \in W^1_{\rm loc}(M)$, then for $f \in D(Q_{\rho,0}) = W^1_0(M,\rho)$ we have
 $$Q_{\rho,0}(f) = \int_M d\Gamma (f).$$
 The {\em intrinsic metric} associated with   $Q_{\rho,0}$ on $L^2(M,\rho \mu)$ is defined by
 \begin{align*}
  d_\rho(x,y) =& \sup\{ |f(x) - f(y)| \mid f \in W^1_{\rm loc}(M) \cap C(M) \text{ with } d\Gamma(f) \leq \rho d\mu  \}\\
  =& \sup\{ |f(x) - f(y)| \mid f \in W^1_{\rm loc}(M) \cap C(M) \text{ with } |\nabla f|^2 \leq  \rho   \}.
 \end{align*}
 Rademacher's theorem implies that for $\rho = 1$ the metric $d_1$ is the geodesic distance on $(M,g)$ and that for general $\rho$ the function $d_\rho$ is a metric, which induces the original topology on $M$. 
 
 We fix some $o \in M$ and let $B^{\rho}_r = \{x \in M \mid d_\rho(x,o) \leq r\}$. Moreover, we denote the volume of $B_r^\rho$ with respet to $\rho \mu$ by
 $$\nu_\rho(r) = (\rho\mu)(B^\rho_r) =  \int_{B^{\rho}_r} \rho d\mu.$$
We recall \cite[Theorem~4]{Stu1}, which in our situation takes the following form.
 \begin{lemma}[Sturm's volume growth test for conservativeness]\label{lemma:sturms test}
  Suppose  that open balls with respect to $d_\rho$ are relatively compact. Then 
 $$\int_1^\infty \frac{r}{\max\{\log \nu_\rho(r),1\} } dr = \infty $$
  implies that $Q_{\rho,0}$ is conservative.
 \end{lemma}

 With this at hand, we can now state and proof the main result of this section.
 
 \begin{theorem}
     If open balls with respect to $d_\rho$ are relatively compact, there exists a smooth potential $V \geq 0$ such that $\Qr$ is conservative in the generalized sense. 
     
     In particular, the assertion holds if the weighted manifold is complete and $\inf_{x \in M} \rho(x) >0$.
 \end{theorem}
 \begin{proof}
   We fix $o \in M$. In this proof we denote by $d$ the geodesic distance induced by $g$ and by $B_r$ the closed $r$ ball around $o$ with respect to $d$. 
 
  For the `In particular'-part note that $c=\inf_{x \in M}\rho(x) >0$ implies $d_\rho \geq d_c$ and so 
  $B_r^\rho \subseteq B_r^c \subseteq B_{r/c} $, for all $r > 0$. Since by the Hopf-Rinow theorem the completeness of $(M,d)$ implies that balls with respect to $d$ are relatively compact, we obtain that open balls with respect to $d_\rho$ are relatively compact.
 
  According to Theorem~\ref{theorem:criteria for generalized conservation} it suffices to show that there exists $V \geq 0$ such that the form  $Q_{\rho + V,0}$ on $L^2(M,(\rho + V) \mu)$ is conservative. We employ Sturm's volume growth test for conservativeness to verify the existence of such a potential. Since $V \geq 0$ it is readily verified that $d_{\rho + V} \geq d_\rho$. In particular, if $d_\rho$-balls are relatively compact, so are $d_{\rho  + V}$-balls. Therefore, we only need to construct $V$ that satisfies Sturm's volume growth condition.
  
  Let $f:\R \to [0,\infty)$ be smooth and convex  such that there exists $r_0>0$ with $f(t) = 0$ for $t \leq r_0$ and such that $[r_0,\infty) \to [0,\infty), t \mapsto f(r)$ is strictly increasing. Since the geodesic distance $d$ has the Lipschitz constant equals 1,  
  Rademacher's theorem implies $|\nabla d(o,\cdot)| \leq 1$ and so the function 
  $$\psi:M \to [0,\infty), \psi(x) = f(d(o,x))$$ 
satisfies $| \nabla \psi (x)|^2 \leq f'(d(o,x))^2$ for almost every $x \in M$.
  %
  %
  We choose a smooth potential $V:M \to [0,\infty)$ such that $V(x) \ge f'(d(o,x))^2$. 
 With the properties of $\psi$ at hand, it follows from the definition of $d_{\rho + V}$ that
  $$d_{\rho + V}(x,o) \geq |\psi(x) - \psi(o)| = f(d(o,x)).$$
  Hence, for $r \geq r_0$ we have the inclusion
  $$B_r^{\rho + V} \subseteq B_{f^{-1}(r)},$$
  which leads to 
  $$\nu_{\rho + V}(r) \leq \int_{B_{f^{-1}(r)}} (\rho + V) d\mu \leq \int_{B_{f^{-1}(r)}} \rho d\mu + f'(f^{-1}(r))^2 \mu(B_{f^{-1}(r)}).$$
  For the last inequality we used the convexity of $f$. Using this and substituting $s = f^{-1}(r)$ we obtain
  \begin{align*}
   \int_{r_0}^\infty \frac{r}{\max\{\log \nu_{\rho + V}(r),1\} } dr &\geq \int_{r_0}^\infty \frac{r}{\max\{\log ((\rho\mu)(B_{f^{-1}(r)}) +  f'(f^{-1}(r))^2 \mu(B_{f^{-1}(r)}))  ,1\} } dr\\
     &=  \int_{f^{-1}(r_0)}^\infty \frac{f(s)f'(s)}{\max\{\log ((\rho\mu)(B_{s}) +  f'(s)^2 \mu(B_s))  ,1\} } ds.
  \end{align*}
  It is now a simple exercise to construct a function $f$ with the required properties such that the latter integral diverges. For example, one can choose $f$ such that  $f(r) = e^{g(r)}$ for $r$ large enough and some sufficiently large smooth convex function $g$. Now the claim follows from Sturm's volume growth criterion.  
 \end{proof}

 \begin{remark}
  For graphs it is known that one can always add a large enough potential to force the corresponding form to be conservative in the generalized sense, see \cite{KL}.
  
  We believe that   our theorem also holds for arbitrary manifolds, i.e., when balls with respect to $d_\rho$ need not be relatively compact. In this case, it should be possible to find $V$ such that $d_{\rho + V}$ has relatively compact balls. With this at hand, one can then argue as in the proof above.
 \end{remark}
 
 \begin{corollary} There exists a complete but stochastically incomplete Riemannian manifold (i.e. $Q_{1,0}$ is not conservative)  and a potential $V$ such that $Q_{1,V}$ is conservative in the generalized sense.
 \end{corollary}
 \begin{proof}
  This follows from the existence of a complete but stochastically incomplete Riemannian manifold  and our previous theorem.
 \end{proof}

\appendix

\section{Local regularity theory} \label{section:local regularity theory}

In this appendix we collect local regularity results for the operator $\cL$. They are consequences of the well-known local elliptic and local parabolic regularity theory in Euclidean spaces. In order to obtain versions on weighted manifolds, one just needs to localize the operators accordingly. For the case when $V = 0$ and $\rho = 1$, this can be found e.g. in \cite[Chapters~6 and 7]{Gri}. Since we assume  $V \geq 0$, the proofs given there can be carried trough verbatim in our situation (otherwise some slight modifications would be needed). In other words, for the reader who is well-acquainted with local regularity theory the following lemmas are simple exercises, while other readers may find the proofs in \cite{Gri}.

The following Sobolev embedding theorem is a version of \cite[Theorem~7.1]{Gri}. 

\begin{lemma}\label{lemma: local sobolev embeddings}
Let $f \in L^2_{\rm loc}(M)$.  If for each $m \geq 0$ we have $(\cL)^m f \in L^2_{\rm loc}(M)$, then $f \in C^\infty(M)$.  Moreover, if $(f_n)$ is a sequence in  $L^2_{\rm loc}(M)$ such that for each $m \geq 0$ also $(\cL)^m f_n \in L^2_{\rm loc}(M)$ and $(\cL)^m f_n \to (\cL)^m f$ in $L^2_{\rm loc}(M)$, then $f_n \to f$ locally uniformly.
\end{lemma}

The following hypoellipticity statement is a version of \cite[Theorem~7.4]{Gri}.

\begin{lemma}\label{lemma:parabolic regularity}
 Let $f \in C^\infty(M)$ and let $u \in \mathcal{D}'((0,T) \times M)$  satisfy
 $$\partial_t u - \cL u = f.$$
 Then $u \in C^\infty((0,T) \times M)$.
\end{lemma}

%
%

\section{Convergence of Semigroups}

Let $\Omega \subseteq M$ an open subset. We let $(T_t^\Omega)$ denote the $L^2(\Omega,\rho\mu)$ semigroup of the Dirichlet form 
$$Q^\Omega_{\rho,V}(f,g) = \int_\Omega \as{\nabla f,\nabla g} d \mu + \int_\Omega Vfg d\mu$$
with domain $D(Q^\Omega_{\rho,V}) = W_0^1(\Omega,\rho + V)$. We extend it to $f \in L^2(M,\rho\mu)$ by letting
$$T_t^\Omega f := \begin{cases} T^\Omega_t f|_\Omega &\text{on } \Omega\\
                   0 &\text{on } M\setminus \Omega.
                  \end{cases}
$$
Similarly, we denote by $(G_\alpha^\Omega)$ the resolvent of $Q^\Omega_{\rho,V}$ extended by $0$ outside of $\Omega$.

\begin{lemma}\label{lemma:convergence of semigroups}
 Let $(\Omega_n)$ an ascending sequence of open subsets of $M$ with $\bigcup_n \Omega_n = M$. Then, for every $t > 0$ and $\alpha>0$  we have $T_t^{\Omega_n} \to \Tr$ and $G_\alpha^{\Omega_n} \to \Gr$ strongly in $L^2(M)$, as $n \to \infty$.
\end{lemma}
\begin{proof}
We prove that $\Qr^{\Omega_n}$ converges to $\Qr$ in the generalized Mosco sense, see. \cite[8~Appendix]{CKK} for a definition. The desired statement then follows from \cite[Theorem~8.3]{CKK}.

We denote by $\pi_n:L^2(M,\rho \mu) \to L^2(\Omega_n,\rho \mu)$ the restriction $f \mapsto f|_{\Omega_n}$. Its adjoint $E_n := \pi_n^*$ extends a function in  $L^2(\Omega_n,\rho \mu)$ to $M$ by letting it equal to $0$ outside of $\Omega_n$. For verifying generalized Moso convergence, we need to prove the following statements.
\begin{itemize}
 \item[(a)] For $f_n \in L^2(\Omega_n,\rho \mu)$, $f \in L^2(M,\rho \mu)$ with $E_n f_n \to f$ weakly in $L^2(M,\rho \mu)$ the inequality
 $$\Qr(f) \leq \liminf_{n \to \infty}Q^{\Omega_n}_{\rho,V}(f_n)  $$
 holds (with the convention $\Qr(g) = \infty$ if $f \not \in D(\Qr)$).
 \item[(b)] For every $f \in D(\Qr)$ there exist $f_n \in D(\Qr^{\Omega_n})$ with $E_n f_n \to  f$ strongly in $L^2(M,\rho\mu)$ and
 $$\limsup_{n\to \infty} \Qr^{\Omega_n}(f_n) \leq \Qr(f).$$
\end{itemize}
The closedness of $\Qr$ implies that it is lower semicontinuous with respect to weak convergence. Hence, $E_n f_n \to f$ weakly in $L^2(M,\rho \mu)$ yields
$$\Qr(f) \leq \liminf_{n\to \infty}\Qr(E_n f_n).$$
It follows from the definition of $\Qr^{\Omega_n}$ that $\Qr(E_n f_n) = Q^{\Omega_n}_{\rho,V}(f_n)$. This proves (a).

For proving (b) we use that $C_c^\infty(M)$ is dense in $D(\Qr)$ with respect to the form norm. For given $f \in D(\Qr)$ let $(f_n)$ a sequence in $C_c^\infty(M)$ that converges to $f$ with respect to $\|\cdot\|_{W^1}$. From this sequence we can build up a sequence $(g_n)$  with ${\rm supp}\, g_n \subseteq \Omega_n$ and $g_n \to f$ with respect to $\|\cdot\|_{W^1}$ as follows.  For $n \in \N$ we define 
$$k_n := \max \{k\leq n \mid {\rm supp}\, f_k \subseteq \Omega_n\}$$
and set $g_n := f_{k_n}$. By construction we have  ${\rm supp}\, g_n \subseteq \Omega_n$. The sequence $(k_n)$ is increasing and since the $(f_n)$ have compact support, it diverges. We obtain $g_n \to f$ with respect to $\|\cdot\|_{W^1}$.  These considerations imply $\pi_n g_n \in D(Q^{\Omega_n}_{\rho,V})$ and
$$\limsup_{n\to \infty} Q^{\Omega_n}_{\rho,V}(\pi_n g_n) = \limsup_{n\to \infty} \Qr(g_n) = \Qr(f).$$
This finishes the proof.
\end{proof}

\section{Two lemmas on measurable choices}

The following lemmas  are certainly well known to experts. Since we could not find proper references, we include their  proofs for the convenience of the reader. Let $0 < T \leq \infty$. By $\lambda$ we denote the Lebesgue measure on $(0,T)$. Let $u:(0,T) \to  L^2(M,\mu)$. A measurable   function $\tilde u \in L_\mathrm{loc}^1((0,T) \times M,\lambda \otimes \mu)$ such that  $u(t) = \tilde u_t$ in $L^2(M,\mu)$ for $\lambda$-a.e. $t \in (0,T)$ is called a {\em locally integrable version of $u$.} Note that by Fubini's theorem this is well-defined, i.e., it is independent of the choice of the representative of $\tilde u$.

\begin{lemma}\label{lemma:measurable choice} 
 Let $0 < T \leq \infty$ and let $u:(0,T) \to L^2(M,\mu)$ continuous.  Then there exists a locally integrable version of $u$.
\end{lemma}
\begin{proof}
 Since the Borel-$\sigma$-algebra of $M$ is countably generated, $L^2(M,\mu)$ is separable. Let $(f_k)_{k \geq 1}$ be a countable orthonormal basis for $L^2(M,\mu)$. Moreover, let $I_n \subseteq (0,T)$ increasing compact intervals with $\bigcup_n I_n = (0,T)$.  For $k,n \in \N$ the maps 
 $$g_{n,k} : (0,T) \times M \to \R,\, (x,t) \mapsto \as{u(t),f_k} 1_{I_n}(t) f_k(x)$$
 are clearly measurable. We consider
 $$u_{n,l}:= \sum_{k = 1}^l g_{n,k}.$$
  Parseval's inequality in  $L^2(M,\mu)$ and the strong continuity of  $u$ imply
 $$\int_{0}^T \int_M |u_{n,l}|^2 d \mu d\lambda \leq \int_{0}^T 1_{I_n}(t) \|\tilde u(t)\|^2 d \lambda(t) < \infty,   $$
 uniformly. Hence, for each $n \in \N$ the limit $u_n := \lim_{l\to \infty} u_{n,l}$ exists in $L^2((0,T) \times M,\lambda \otimes \mu)$. For $n \geq m$ the functions $u_n$ and $u_m$ only differ on $I_n \setminus I_m$; indeed we have $u_m = u_n 1_{I_m \times M}$. Hence, the limit $\tilde u = \lim_{n\to \infty} u_{n}$ exists in $L^1_{\rm loc}((0,T) \times M,\lambda \otimes \mu)$  and satisfies $\tilde u1_{I_n \times M} = u_n$.   Parseval's identity and the properties of $\tilde u$ yield
 \begin{align*}
  \int_0^T \|\tilde u_t- u(t)\|^2 d \lambda(t) &= \lim_{n\to \infty} \int_{I_n} \|\tilde u_t - u(t)\|^2 d \lambda(t) \\
  &= \lim_{n\to \infty} \int_{I_n} \|u_n (t,\cdot) - \tilde u(t)\|^2 d \lambda(t)\\
  &= \lim_{n\to \infty}\lim_{l\to \infty} \int_{I_n} \|(u_{n,l})_t - u(t)\|^2 d\lambda(t)\\
  &= \lim_{n \to \infty} \lim_{l\to \infty} \int_{I_n} \sum_{k = l+1}^\infty |\as{ u(t),f_k}_2|^2 d\lambda (t).
 \end{align*}
 The strong continuity of $u$ yields
 $$\sup_{t \in I_n} \sum_{k = l+1}^\infty |\as{u(t),f_k}_2|^2 \leq \sup_{t \in I_n} \|u(t)\|^2_2 < \infty, \text{ for all } l \in \N.$$
 Hence, Lebesgue's theorem yields
 $$ \lim_{l\to \infty} \int_{I_n} \sum_{k = l+1}^\infty |\as{\tilde u(t),f_k}_2|^2 d \lambda (t) = 0,$$
 and the assertion $\tilde u_t = u(t)$ in $L^2(M,\mu)$ for $\lambda$-a.e. $t \in (0,T)$ is proven. 
\end{proof}

We say that $u:(0,T) \to L^2(M,\mu)$ is continuously differentiable if for all $0 < s < T$ the limit
$$u'(s) = \lim_{h \to 0} \frac{1}{h}(u(s + h) - u(s))$$
exists in $L^2(M,\mu)$ and $u':(0,T) \to L^2(M,\mu)$ is continuous. In the following lemma $\partial_t$ denotes the distributional time derivative on $\mathcal{D}'((0,\infty) \times M)$.

\begin{lemma}\label{lemma:measurable choice2}
 Let $0 < T \leq \infty$ and let $u:(0,T) \to L^2(M,\mu)$ be continuously differentiable.  Then there exists a locally integrable version $\tilde u$ of $u$ such that $\partial_t \tilde u \in L^1_\mathrm{loc}((0,T) \times M,\lambda \otimes \mu)$ is a locally integrable version of $u'$. 
\end{lemma}
 \begin{proof}
  According to the previous lemma $u$ and $u'$ have locally integrable versions $\tilde u$ and $v$, respectively. Hence, it suffices to prove $v = \partial_t \tilde u$. For $\varphi \in \mathcal{D}((0,\infty)\times M)$ we compute
 \begin{align*}
  \as{\tilde u,\partial_t \varphi} &=   \int_0^\infty \int_M \tilde u_s(x) (\partial_t \varphi)_s(x) d\mu(x) d\lambda(s)   \\
  &= \lim_{h\to 0} h^{-1}  \int_0^\infty \int_M  \tilde u_s(x) (\varphi_{s + h}(x) - \varphi_s(x))d\mu(x) d \lambda(s) \\
  &=  \lim_{h\to 0} h^{-1}  \int_0^\infty \int_M (\tilde u_{s-h}(x) - \tilde u_s(x)) \varphi_s(x)d \mu(x) d\lambda(s) \\
  &= \lim_{h\to 0} h^{-1} \int_0^\infty  \as{u(s-h) - u(s),\varphi_s}_\mu  d \lambda(s) \\
  &= - \int_{0}^\infty  \as{u'(s),\varphi_s}_\mu d\lambda (s).   
 \end{align*}
  For the second to last equality we used that $\tilde u$ is a locally integrable version of $u$. Moreover, for the last equality we used a standard result for differentiating under the integral sign using that $u$ is continuously differentiable and $\varphi$ has compact support in $(0,T) \times M$. Since $v$ is a locally integrable version of $u'$, we further obtain
  $$\int_{0}^\infty  \as{u'(s),\varphi_s}_\mu d\lambda (s) = \int_{0}^\infty  \as{v,\varphi_s}_\mu d\lambda (s) = \as{v,\varphi}.$$
  This proves $\partial_t \tilde u = v$, as by definition $\as{\partial_t \tilde u, \varphi} = - \as{\tilde u,\partial_t \varphi}$. 
 \end{proof}


 \bibliographystyle{plain}
 
\bibliography{literatur}

\end{document}